\documentclass{amsart}
\usepackage{latexsym}
\usepackage{amsmath,amsfonts,amssymb,amscd}
\usepackage[backref=page]{hyperref}
\usepackage{amsxtra}
 \numberwithin{equation}{section}
\newtheorem{theo}{Theorem}[section]
\newtheorem{lemma}[theo]{Lemma}
\newtheorem{corol}[theo]{Corollary}
\newtheorem{prop}[theo]{Proposition}
\theoremstyle{definition}
\newtheorem{remark}[theo]{Remark}

\newtheorem{example}[theo]{Example}

\newtheorem{defi}[theo]{Definition}

\newcommand{\Real}{\mathbb R}
\newcommand{\Complex}{\mathbb C}
\newcommand{\Kapa}{\mathbb K}
\newcommand{\Nat}{\mathbb N}
\newcommand{\norm}{\|\cdot\|}
\newcommand{\Lra}{\Longrightarrow}
\newcommand{\Ra}{\Longrightarrow}
\newcommand{\lra}{\longrightarrow}

\newcommand{\sms}{\smallsetminus}
\newcommand{\sse}{\subseteq}

\newcommand{\rd}{\mathcal}
\newcommand{\inter}{\operatorname{int}}

\newcommand{\clos}{\operatorname{cl}} %

\newcommand{\conv}{\operatorname{co}}

\newcommand{\Lip}{\operatorname{Lip}} %

\newcommand{\vphi}{\varphi}
\newcommand{\epsic}{\varepsilon}

\newcommand{\mbx}{\mbox}

\newcommand{\ba}{\begin{array}}
\newcommand{\ea}{\end{array}}
\newcommand{\ben}{\begin{enumerate}}
\newcommand{\een}{\end{enumerate}}
\newcommand{\eite}{\end{itemize}}
\newcommand{\bite}{\begin{itemize}}
\newcommand{\bequ}{\begin{equation}} %
\newcommand{\eequ}{\end{equation}} %
\newcommand{\bequs}{\begin{equation*}} %
\newcommand{\eequs}{\end{equation*}} %
\newcommand{\beqs}{\begin{equation*}} %
\newcommand{\eeqs}{\end{equation*}} %
\newcommand{\bc}{\begin{center}}
\newcommand{\ec}{\end{center}}
\newcommand{\bfr}{\begin{flushright}}
\newcommand{\efr}{\end{flushright}}

\newcommand{\lbda}{\lambda}
\newcommand{\Lbda}{\Lambda}

\begin{document}

\title{b-metric   spaces,  fixed points and Lipschitz functions}
\author{S. Cobza\c{s} }
\address{S. Cobza\c{s},  Babe\c{s}-Bolyai University,
Department of Mathematics,
Cluj-Napoca, Romania}
\email{scobzas@math.ubbcluj.ro}

\date{\today}
\begin{abstract}
  The paper is concerned with b-metric and generalized b-metric spaces. One proves the existence of the completion of a generalized b-metric space and some fixed point results. The behavior of Lipschitz functions on b-metric spaces of homogeneous type, as well as of Lipschitz functions defined on,  or with values in  quasi-Banach spaces, is studied. \medskip


  \end{abstract}

\maketitle
\tableofcontents
\section*{Introduction}

There are a lot of extensions of the notions of metric and metric space -- see for instance the books \cite{Deza}, \cite{Kirk-Shah}, \cite{Rus-PP},  or the survey papers \cite{beri-chob13}, \cite{khamsi15}. In this   paper  we concentrate  on b-metric and generalized b-metric spaces, with emphasis on their topological properties, some fixed point results and Lipschitz functions on such spaces.

A part of the results from this paper are included  in \cite{cobz-czerw18}.

\section{b-metric spaces}

In this section we  present some results on b-metric spaces.
\subsection{Topological properties and metrizability}\label{Ss.bms-top}
 A \emph{b-metric} on a nonempty set $X$ is a function $d:X\times X\to[0,\infty)$ satisfying the conditions
 \begin{equation}\label{def.b-metric}
 \begin{aligned}
   {\rm (i) } \quad &d(x,y) = 0 \iff x=y;\\
    {\rm (ii) } \quad &d(x,y) = d(y,x);\\
   {\rm (iii) } \quad &d(x,y) \le s[d(x,z) + d(z,y)],
\end{aligned}
\end{equation}
for all $x,y,z\in X,$  and for some fixed number $s\ge 1.$ The pair $(X,d)$ is called a b-\emph{metric space}. Obviously, for $s=1$ one obtains a metric on $X$.

\begin{example}
If  $(X,d)$ is a  metric space and  $\beta >1$, then $d^\beta(x,y)$ is a b-metric.
\end{example}
Indeed,
\begin{align*}
 d^\beta(x,y)&\le [d(x,z)+d(z,y)]^\beta\le 2^\beta \left(\max\{d(x,z),d(z,y)\}\right)^\beta\\&\le 2^\beta[d^\beta(x,y)+d^\beta(x,y)]\,.
\end{align*}

The $s$-relaxed triangle inequality implies
\begin{equation}\label{s-relax-n}
d(x_0,x_n)\le sd(x_0,x_1)+s^2d(x_1,x_2)+\dots+s^{n-1}d(x_{n-2},x_{n-1})+s^{n-1}d(x_{n-1},x_n)\,,
\end{equation}
for all $n\in\Nat$ and all $x_0,x_1,\dots,x_n\in X$.

Indeed, we obtain successively
\begin{align*}
  d(x_0,x_n)&\le sd(x_{0},x_1)+sd(x_{1},x_n)\le sd(x_{0},x_1)+s^2d(x_{1},x_2)+s^2d(x_2,x_n)\le \dots\\
 &\le sd(x_0,x_1)+s^2d(x_1,x_2)+\dots+s^{n-1}d(x_{n-2},x_{n-1})+s^{n-1}d(x_{n-1},x_n)
\end{align*}

Along with the inequality (iii), called the $s$-\emph{relaxed triangle inequality}, one considers also the $s$-\emph{relaxed polygonal inequality}

\begin{equation}\label{def.polyg}\tag{{\rm iv}}
  d(x_0,x_n) \le s[d(x_0,x_1) + d(x_1,x_2)+\dots+d(x_{n-1},x_n)],
\end{equation}
for all $x_0,x_1,\dots,x_n\in X$ and all $n\in\Nat.$

For $n=2$ one obtains the inequality (iii). The following example shows that the converse is not true -- there exist b-metrics that do not satisfy the relaxed polygonal inequality.

\begin{example}[\cite{Kirk-Shah}, Theorem 12.10]\label{ex.polyg}  Let $X=[0,1] $ and   $d(x,y)=(x-y)^2,\, x,y\in [0,1].$ Then $d$ is a 2-relaxed metric on $X$ which is not polygonally $s$-relaxed for any $s\ge 1.$
     \end{example}

     Indeed, it is easy to check that $d$ satisfies the 2-relaxed triangle inequality. Suppose that, for some $s\ge 1,$  $d$ satisfies the  $s$-relaxed polygonal  inequality. Taking $x_i=\frac in,\, 1\le i\le n-1,$ we obtain
     $$
    \frac1s= \frac 1s\cdot d(0,1)\le  d(0,x_1)+d(x_1,x_2)+\dots+d(x_{n-1},1)=n\cdot\left(\frac1n\right)^2=\frac1n\,,$$
     for all $n\in \Nat,$ which is impossible.

One can consider also an ultrametric version of (iii):
\begin{equation}\label{b-ultra}\tag{{\rm iii$'$}}
  d(x,y)\le\lbda\max\{d(x,z),d(y,z)\}\,,
\end{equation}
for all $x,y,z\in X$. It is obvious that
\begin{align*}
  \mbx{(iii$'$)}\;&\Lra\;   \mbx{(iii) \;\, with }\; s=\lbda;\\
\mbx{(iii)}\;&\Lra\; \mbx{(iii$'$) \; with }\; \lbda=2s.
 \end{align*}

The  condition
\begin{equation}\label{b-metric-eps}\tag{{\rm iii$''$}}
  \max\{d(x,z),d(y,z)\}\le\epsic \;\Lra\;  d(x,y)\le 2\epsic\,,
\end{equation}
for all $\epsic >0$ and $x,y,z\in X$,  is equivalent to \eqref{b-ultra} with  $\lbda =2$.

Let now $(X,d)$ be again a b-metric space.
One introduces a topology on a b-metric space $(X,d)$ in the usual way. The ``open" ball $B(x,r)$ of center $x\in X$ and radius $r>0$ is given by
\begin{equation*}
B(x,r)=\{y\in X : d(x,y)<r\}\,.
\end{equation*}

A subset $Y$ of $X$ is called open if for every $x\in Y$ there exists a number $r_x>0$ such that $B(x,r_x)\sse Y.$ Denoting by $\tau_d$ (or  $\tau(d)$) the family of all open  subsets of $X$ it follows that $\tau_d$ satisfies the axioms of a topology. This topology is derived from a uniformity $\rd U_d$ on $X$ having as basis the sets
$$
U_\epsic=\{(x,y)\in X\times X  : d(x,y)<\epsic\},\quad \epsic >0\,.$$

The  uniformity $\rd U_d$ has a countable basis $\{U_{1/n} : n\in\Nat\}$ so that, by Frink's metrization theorem (\cite{frink37}),  the uniformity $\rd U_d$ is derived from a metric $\rho$, hence  the topology $\tau_d$ as well. This was remarked in the paper
\cite{maci-sego79a}. In \cite{fagin03} it is shown that the topology $\tau_d$ satisfies the hypotheses of the Nagata-Smirnov metrizability theorem.

Concerning  the metrizability of uniform and topological spaces, see the treatise \cite{Engel}.

There exist  also  direct proofs of the metrizability of the topology of a b-metric space.

Let $(X,d)$ be a b-metric space. Put
\begin{equation}\label{def.1-metric}
\rho(x,y)=\inf\Big\{\sum_{k=0}^nd(x_{i-1},x_i)  \Big\}\,,\end{equation}
where the infimum is taken over all $n\in\Nat$ and all chains $x=x_0,x_1,\dots, x_n=y$ of elements in $X$ connecting $x$ and $y$.

As remarked Frink \cite{frink37}, if a b-metric $d$ satisfies \eqref{b-ultra} for $\lambda =2$, then   formula  \eqref{def.1-metric} defines a metric  equivalent to $d$.  We present the result in the form given by Schroeder \cite{schroed06}.

\begin{theo}[A.~H. Frink \cite{frink37} and V.~Schroeder \cite{schroed06}] \label{t.Frink} If $d:X\times X\to [0,\infty)$ satisfies the conditions {\rm(i), (ii)} from \eqref{def.b-metric} and \eqref{b-ultra} for some $1\le\lbda \le 2,$ then the
  function $\rho$ defined by \eqref{def.1-metric}  is a metric on $X$ satisfying the inequalities $\frac1{2\lbda} d\le \rho\le d$.
\end{theo}

V.~Schroeder \cite{schroed06}  also  showed that for every $\epsic>0$ there exists a b-metric $d$ satisfying \eqref{def.b-metric}.(iii) with $s=1+\epsic$ such that the mapping $\rho$ defined  by \eqref{def.1-metric}   is not a metric. Other example showing the limits of Frink's metrization method was given in  {An} and   {Dung} \cite{an-dung15b}.

General results of   metrizability  were obtained in \cite{aimar98} and \cite{stempak09} by a slight modification of Frink's technique.

Let $(X,d)$ be a b-metric space.  For $0<p\le 1$ define
\begin{equation}\label{def.p-metric}
\rho_p(x,y)=\inf\Big\{\sum_{k=0}^nd(x_{i-1},x_i)^p \Big\}\,,\end{equation}
where the infimum is taken over all $n\in\Nat$ and all chains $x=x_0,x_1,\dots, x_n=y$ of elements in $X$.

The function $ \rho_p$ defined by \eqref{def.p-metric} satisfies the conditions
 \begin{enumerate}
 \item $\rho_p(x,y)=\rho_p(y,x),$
 \item  $\rho_p(x,y)\le\rho_p(x,z)+\rho_p(z,y),$
 \item  $d^p(x,y)\le\rho_p(x,y)$\,,
\end{enumerate}
 for all $x,y,z\in X$, i.e., $\rho$ is a pseudometric on $X$ and $d^p$ is dominated by $\rho$.

  \begin{theo}[\cite{stempak09}]\label{t.Stemp} Let  $d$ be a b-metric on a nonempty set $X$ satisfying the  $s$-relaxed triangle  inequality \eqref{def.b-metric}.(iii), for some   $s\ge 1.$ If the number $p\in (0,1]$ is given by the equation $(2s)^p=2$, then the mapping $\rho_p:X\times X \to [0,\infty)$ defined by \eqref{def.p-metric} is a metric on $X$ satisfying the   inequalities
\begin{equation}\label{ineq1.Stemp}
  \rho_p(x,y)\le d^p(x,y)\le 2  \rho_p(x,y)\,,
  \end{equation}
  for all $x,y\in X$.

  The same conclusions hold if $d$ satisfies   the conditions (i), (ii) from \eqref{def.b-metric} and \eqref{b-ultra} for some $\lbda \ge 2.$ In this case  $0<p\le 1$ is given by $\lbda^p=2$ and the metric $\rho_p$ satisfies the inequalities
  \begin{equation}\label{ineq2.Stemp}
  \rho_p(x,y)\le d^p(x,y)\le 4  \rho_p(x,y)\,,
  \end{equation}
  for all $x,y\in X$.
\end{theo}

The inequalities \eqref{ineq1.Stemp} have the following consequences.
\begin{corol}  Under the hypotheses of Theorem \ref{t.Stemp},   $\tau_d=\tau_\rho,$ that is, the topology of any b-metric space is metrizable, and
the convergence of sequences with respect to  $\tau_d$ is characterized in the following way:
\begin{equation*}
x_n\xrightarrow{\tau_d} x \iff d(x,x_n)\lra 0 \,,
\end{equation*}
for any sequence $(x_n)$ in $X$ and  $x\in X$.
\end{corol}\begin{proof}
 The equality of topologies follows from the inclusions
 $$
 B_d(x,r^{1/p})\sse B_\rho(x,r)\;\mbx{ and }\; B_\rho\big(x, 4^{-1}r^p\big)\sse B_d(x,r)\,,$$
valid for all $x\in X$ and $r>0$.

The statement concerning sequences is a consequence of the equality $\tau_d=\tau_\rho$ and of the inequalities \eqref{ineq1.Stemp}.
\end{proof}

\begin{remark}
  In \cite{aimar98}  the proof is given for a $p$ satisfying the inequalities
  \begin{equation}\label{ineq1.Aimar}
  1\ge p\ge \left(\log_2(3s)\right)^{-2},\end{equation}
  while
  in Theorem \ref{t.Stemp} the  result holds for
  \begin{equation}\label{ineq2.Aimar}
  p= \left(\log_2(2s)\right)^{-1}.\end{equation}
  Putting
  $$
  \tilde \rho(p)=\sup\{p\in(0,1] : \rho_p\;\mbox{is a metric, Lipschitz equivalent to}\; d^p\}\,,$$
the estimation \ref{ineq1.Aimar}   yields $\tilde\rho(p)\ge  \left(\log_2(3s)\right)^{-2}$, while from \ref{ineq2.Aimar} one obtains the better evaluation
$\tilde\rho(p)\ge  \left(\log_2(2s)\right)^{-1}$, which cannot be improved,  as it is shown by the example of the spaces $\ell^p$ with $0<p<1$.

  A proof of Theorem \ref{t.Stemp} is also given in the book by Heinonen \cite[Prop. 14.5]{Heinonen}, with the evaluation  $\,p\ge (\log_2\lbda)^{-2},$ where $\lbda$ is the constant from \eqref{b-ultra}.

\end{remark}
\begin{remark}\label{re.p-norm}
  It follows that $  \tilde\rho(x,y)=\rho_p(x,y)^{1/p},\, x,y\in X$, is a b-metric on $X$, Lipschitz equivalent to $d$ and satisfying the inequality
$$
\tilde\rho(x,y)^p\le\tilde\rho(x,z)^p+\tilde\rho(z,y)^p\,,$$
for all $x,y,z\in X$. This is  a well known fact in the theory of quasi-normed spaces, where a quasi-norm $\norm$ satisfying the inequality
$$
\|x+y\|^p\le \|x\|^p+\|y\|^p,$$
for some $0<p\le 1$ is called a $p$-norm (see Subsection \ref{Ss.quasi-normed space}).
\end{remark}

Let $(X,d)$ be a b-metric space. The b-metric $d$ is called
\bite
\item \emph{continuous} if
\begin{equation}\label{eq1.cont-metric}
d(x_n,x)\to 0\;\mbox{ and }\; d(y_n,y)\to 0\;\Longrightarrow d(x_n,y_n)\to d(x,y)\,,
\end{equation}
\item   \emph{separately continuous} if  the function $d(x,\cdot)$ is continuous on $X$ for every $x\in X$, i.e.,
\begin{equation}\label{eq2.sep-cont-metric}
  d(y_n,y)\to 0\;\Longrightarrow d(x,y_n)\to d(x,y)\,,
\end{equation}
\eite
for all sequences $(x_n),(y_n)$ in $X$ and all $x,y\in X$.

The topology $\tau_d$  generated by a b-metric $d$ has some peculiarities -- a ball $B(x,r)$ need not   be $\tau_d$-open and the b-metric $d$ could not be continuous on
$X\times X$.

\begin{remark}
  Let $(X,d)$ be a b-metric space and $x\in X$. Then
  $$
   B(x,r) \;\mbx{ is $\tau_d$-open for every } r>0 \iff d(x,\cdot) \;\mbx{ is upper semicontinuous on }\;  X.$$

   Consequently, if the b-metric is separately continuous on $X$, then the balls $B(x,r)$ are $\tau_d$-open.
\end{remark}

The equivalence follows from the equality
$$
B(x,r)=d(x,\cdot)^{-1}\big((-\infty,r)\big)\,.$$

We present now an  example of a b-metric space where the balls are not necessarily open.

\begin{example}[\cite{stempak09}] Consider  a fixed number  $\epsic >0$. For  $X=\Nat_0=\{0,1,\dots\} $  let $d:X\times X\to[0,\infty)$ be defined by
$$\begin{matrix}\label{ex.Stemp}
  &d(0,1)=1,  &d(0,m)=1+\epsic &\;\mbx{ for } m\ge 2\\
  &d(1,m)=\frac1m,  &\;\;d(n,m)=\frac1n+\frac1m  &\;\mbx{ for } n\ge 2
\end{matrix}$$
and extended to $X\times X$ by $d(n,n)=0$ and symmetry.

Then
$$
d(m,n)\le (1+\epsic)[d(n,k)+d(k,m)]\,,$$
for all $m,n,k\in X$, $B\left(0,1+\frac\epsic2\right)=\{0,1\}$ and the ball $B(1,r)$ contains an infinity of terms for every $r>0$, that is, for any $1\in B\left(0,1+\frac\epsic2\right)$,  $\, B(1,r) \nsubseteq B\left(0,1+\frac\epsic2\right)$ for every $r>0,$ showing that the ball $ B\left(0,1+\frac\epsic2\right)$ is not $\tau_d$-open.
\end{example}

Other examples are given in \cite{an-dung15}.

\begin{remark} If, for some $0<p<1,$   a b-metric $d$ on a set $X$ satisfies the inequality
\begin{equation}\label{ineq.p-metric}
d(x,y)^p\le d(x,z)^p+d(z,y)^p,\;\mbx{ for all }\; x,y,z \in X,
\end{equation}
then the balls corresponding to $d$ are $\tau_d$-open.  Moreover, the b-metric $d$ is continuous.

 By Remark \ref{re.p-norm},  the b-metric   $\tilde\rho=\rho_p^p$ corresponding to the metric    $\rho_p$ constructed in Theorem \ref{t.Stemp} satisfies  the inequality \eqref{ineq.p-metric}.
\end{remark}

Indeed, let $B(x,r)$ be a ball in $(X,d)$ and $y\in B(x,r)$. We have to show that there exists $r'>0$ such that $B(y,r')\sse B(x,r)$.
Taking $r':=\left(r^p-d(x,y)^p\right)^{1/p}>0,$ then $d(y,z)<r'$ implies
\begin{align*}
 d(x,z)^p&\le d(x,y)^p+d(y,z)^p  \\
 &< d(x,y)^p+r'^p=r^p,
\end{align*}
that is, $d(x,z)<r.$  The continuity of the b-metric $d$ follows from the inequality
$$
| d(x_n,y_n)^p- d(x,y)^p|\le  d(x_n,x)^p+ d(y_n,y)^p\,,$$
which can be proved as in the metric case (using \eqref{ineq.p-metric}).

\medskip

\textbf{Equivalence notions for b-metrics.}

In connection to the metrizability of b-metric spaces, we mention the following notions of equivalence for b-metrics.

Let $d_1,d_2$ be two b-metrics on the same set $X$. Then $d_1,d_2$ are called
\bite\item \emph{topologically equivalent} if $\tau_{d_1}=\tau_{d_2}$;
\item \emph{uniformly  equivalent} if the identity mapping $I_X$ on $X$ is uniformly continuous both from $(X,d_1)$ to $(X,d_2)$ as well as from  $(X,d_2)$ to $(X,d_1)$,
i.e.
\begin{align*}
  &\forall \epsic >0, \exists \delta(\epsic)>0\;\mbx{ such that } \forall x,y\in X\;(d_1(x,y)\le\delta(\epsic)\; \Ra\; d_2(x,y)\le \epsic),\\
   &\forall \epsic >0, \exists \delta(\epsic)>0\;\mbx{ such that }\; \forall x,y\in X\;( d_2(x,y)\le\delta(\epsic)\; \Ra\; d_1(x,y)\le \epsic)\,.
\end{align*}
\item \emph{Lipschitz equivalent} if there exist $c_1,c_2>0$ such that
$$c_1d_2(x,y)\le d_1(x,y)\le c_2d_2(x,y)\,,$$
for all $x,y\in X$
\eite

Of course,   the above definitions applies to metrics as well, as particular cases of b-metrics.

\begin{remark} It is obvious that, in general, \medskip

Lipschitz equivalence \; $\Ra$\; uniform  equivalence \; $\Ra$\;  topological equivalence. \medskip

For quasi-norms, topological equivalence is equivalent to Lipschitz equivalence.
\end{remark}

So the expression ``the topology $\tau_d$ generated by a b-metric $d$ on a set $X$ is metrizable" means that there exists a metric $\rho$ on $X$ topologically equivalent to $d$.

The problem of the existence of a metric that is  Lipschitz equivalent to a b-metric was solved in \cite{fagin03}, where this property was called \emph{metric boundedness}.
\begin{theo}[\cite{fagin03}, see also \cite{Kirk-Shah}, Theorem 12.9] Let $(X,d)$ be a b-metric space. Then $d$ is Lipschitz equivalent to a metric  if and only if
$d$ satisfies the $s$-relaxed polygonal inequality   \eqref{def.polyg} for some $s\ge 1.$
  \end{theo}
\subsection{An axiomatic definition of balls in $b$-metric spaces}
H. Aimar \cite{aimar98} found a set of properties characterizing balls in b-metric spaces.

For a nonempty set $X$ consider a mapping $U:X\times (0,\infty)\to \rd P(X)$ satisfying the following properties:
\bequ\label{def.balls}
\begin{aligned}
  {\rm (i)}\quad &\bigcap_{r>0} U(x,r)=\{x\};\\
  {\rm (ii)}\quad &\bigcup_{r>0} U(x,r)=X;\\
  {\rm (iii)}\quad &0<r_1\le r_2\;\Ra\; U(x,r_1)\sse U(x,r_2);\\
  {\rm (iv)}\quad &\mbox{there exists  } \,c\ge 1\,\mbox{ such that } \\
  & y\in U(x,r)\;\Ra\; U(x,r)\sse U(y,cr)\;\mbox{ and }\;U(y,r)\sse U(x,cr),
\end{aligned}\eequ
for all $x\in X$ and $r>0$.

We call the sets $U(x,r)$ \emph{formal balls}.
\begin{remark}\label{re1.formal}
  By (i), $x\in U(x,r)$ for all $r>0$ and $x\in X$.
\end{remark}

It is easy to check that if $(X,d)$ is a $b$-metric space, where the $b$-metric $d$ satisfies the relaxed triangle inequality for some $s\ge 1,$ then
the sets $U(x,r)=B_d(x,r),\, x\in X, r>0$ satisfy the properties from \eqref{def.balls} with $c=2s.$

Conditions (i)--(iii) are easy to check.  For (iv),
if $y\in B_d(x,r)$ and $z\in B_d(x,r)$, then
$$
d(y,z)\le s(d(y,x)+d(x,z))<2sr\,,$$
i.e. $z\in B_d(y,2sr).$

Similarly,  $y\in B_d(x,r)$ and $z\in B_d(y,r)$ imply
$$
d(x,z)\le s(d(x,y)+d(y,z))<2sr\,,$$
i.e. $z\in B_d(x,2sr).$

It can be shown that, conversely,  a family of subsets of $X$ satisfying the   properties from \eqref{def.balls} generates a $b$-metric on $X$, the balls  corresponding to $d$ being tightly connected with the sets $U(x,r).$

\begin{theo}[\cite{aimar98}] \label{t1.Aimar} For  a nonempty set  $X$  and   a mapping $\, U:X\times (0,\infty)\to \rd P(X)$   satisfying the   properties from \eqref{def.balls}, define $d:X\times X\to [0,\infty)$ by
\bequ\label{def.d-U}
d(x,y)=\inf\{r>0 : y\in U(x,r) \,\mbx{ and }\, x\in U(y,r)\},\quad  x,y\in X.\eequ

Then:
\ben
\item $d$ is a $b$-metric on $X$ satisfying the relaxed triangle inequality for $s=c$;
\item the open balls $B_d(x,r)$ corresponding to $d$ and the sets $U(x,r)$ are related by the inclusions
\bequ\label{eq.ball-incl}
U\left(x,(\gamma c)^{-1}r\right)\sse B_d(x,r)\sse U(x,r),\eequ
for all $x\in X$ and $r>0$, where $\gamma >1$.
\een\end{theo}\begin{proof}
  We shall verify only the relaxed triangle inequality. Let $x,y,z\in X $ and $\epsic >0$. By the definition \eqref{def.d-U} of $d$ there exists $r_1,r_2>0$ such that
  \bequ\label{eq1.d-U}
  \begin{aligned}
     & 0<r_1<d(x,z)+\epsic\;\mbx{ and }\; z\in U(x,r_1), \, x\in U(z,r_1);\\
  & 0<r_2<d(z,y)+\epsic\;\mbx{ and }\; z\in U(y,r_2), \, y\in U(z,r_2).
  \end{aligned}\eequ

  Taking into account (iii) it follows
  $$
          x,y\in U(z,r_1+r_2) \; \mbx {and }\;
     z\in U(x,r_1+r_2)\cap   U(y,r_1+r_2)\,.$$

  Applying (iv) to $x\in U(z,r_1+r_2)$ one obtains
  $$ y\in U(z,r_1+r_2)\sse U(x,c(r_1+r_2)).$$

  Similarly, $y\in U(z,r_1+r_2)$ implies
  $$ x\in U(z,r_1+r_2)\sse U(y,c(r_1+r_2)),$$
  so that, by the definition of $d,\, d(x,y)\le c(r_1+r_2)$. But then, by adding the inequalities \eqref{eq1.d-U}, one obtains
  $$
  d(x,y)\le  c(r_1+r_2)<c(d(x,z)+d(z,y))+2c\epsic\,.$$

  Since these hold for all $\epsic >0,$ it follows
   $$
  d(x,y)\le  c(d(x,z)+d(z,y))\,.$$

  Let us prove now the inclusions \eqref{eq.ball-incl}.

  If $d(x,y)<r$, then there exists $0<r'<r$ such that $y\in U(x,r')$ and $x\in U(y,r'). $ By (iii), $U(x,r')\sse U(x,r)$, so that $B_d(x,r)\sse U(x,r).$

  Let now $y\in U(x,(\gamma c)^{-1}r).$  By Remark \ref{re1.formal} and (iv),
\begin{align*}
  &x\in U\left(x,(\gamma c)^{-1}r\right)\sse U\left(y, r/\gamma\right),\; \mbx{and}\\
  &y\in U\left(y,(\gamma c)^{-1}r\right)\sse U\left(x, r/\gamma\right)\,.
\end{align*}

  By   the definition of $d,\, d(x,y)\le r/\gamma<r\,,$ i.e. $y\in B_d(x,r).$
\end{proof}

\begin{remark}
   In \cite{aimar98} it is shown that $d$ satisfies the relaxed triangle inequality with $s=2c.$
   Also, the first inclusion in \eqref{eq.ball-incl} is proved for $\gamma =2,$ i.e.  it is shown that $U\left(y,(2c)^{-1}r\right)\sse B_d(x,r)$.
   \end{remark}

In \cite{aimar10} a similar characterization is given in terms of some subsets of $X\times X$. Denoting by $\Delta$ the diagonal of $X\times X$,
$$\Delta =\{(x,x) : x\in X\}\,,$$
one considers a mapping $V:(0,\infty)\to \rd P(X\times X)$ satisfying the properties:
\bequ\label{def.ant}
\begin{aligned}
  {\rm (j)}\quad &\bigcap_{r>0} V(r)=\Delta;\\
  {\rm (jj)}\quad &\bigcup_{r>0} V(r)=X\times X;\\
  {\rm (jjj)}\quad &0<r_1\le r_2\;\Ra\; V(r_1)\sse V(r_2);\\
  {\rm (jv)}\quad &\mbox{there exists  } \,c\ge 1\,\mbox{ such that } V(r)\circ V(r)\sse V(cr).\\
  \end{aligned}\eequ
for all $r>0$.

By analogy with the case of uniform spaces we call the sets $V(r)$ \emph{antourages}.
If $(X,d)$ is a $b$-metric space with $d$ satisfying the relaxed triangle inequality for some $s\ge 1,$ then the sets
$$
W_d(r)=\{(x,y)\in X\times X : d(x,y)<r\},\quad, r>0,$$
satisfy the conditions from \eqref{def.ant} with $c=2s$. Indeed, the conditions (j)--(jjj) are easily verified. To check (jv), suppose that $(x,y)\in  W_d(r)\circ V(r).$
Then there  exists $z\in X$ such that $(x,z)\in W_d(r)$ and $(z,y)\in W_d(r),$ implying
$$
d(x,y)\le s(d(x,z)+d(z,y))<2sr\,,$$
showing that (jv) holds with $c=2s.$

A converse result holds in this case too.

\begin{theo}[\cite{aimar10}] For  a nonempty set  $X$  and   a mapping $\,V: (0,\infty)\to \rd P(X\times X)$   satisfying the   properties from \eqref{def.ant}, define $d:X\times X\to [0,\infty)$ by
\bequ\label{def.d-unif}
d(x,y)=\inf\{r>0 : (x,y)\in V(r) \},\quad x,y\in X.\eequ

Then:
\ben
\item $d$ is a $b$-metric on $X$ satisfying the relaxed triangle inequality for $s=c$;
\item for any $0<\gamma <1$ the following inclusions hold
\bequ\label{eq1.inclusions}
V\left(\gamma r\right)\sse W_d(r)\sse V(r),\eequ
for all   $r>0$.
\een\end{theo}\begin{proof}
  Again, we prove only the validity of the relaxed triangle inequality. Let $x,y,z\in X$. By the definition of $d$ there exist $r_1,r_2>0$ such that
  \bequ\label{eq1.d-unif}
  \begin{aligned}
     & 0<r_1<d(x,z)+\epsic\;\mbx{ and }\; (x,z)\in V(r_1);\\
  & 0<r_2<d(z,y)+\epsic\;\mbx{ and }\; (z,y)\in V(r_2).
  \end{aligned}\eequ

  By (jjj) $V(r_1)\cup V(r_2)\sse V(r_1+r_2)$ so that, by (jv), $(x,y)\in V(c(r_1+r_2)),$ implying $d(x,y)\le c(r_1+r_2).$

 Consequently, the   inequalities \eqref{eq1.d-U}  yield by addition
  $$
  d(x,y)\le  c(r_1+r_2)<c(d(x,z)+d(z,y))+2c\epsic\,.$$

  Since these hold for all $\epsic >0,$ it follows
   $$
  d(x,y)\le  c(d(x,z)+d(z,y))\,.$$

  The proof of the inclusions \eqref{eq1.inclusions} is simpler than in the case considered in Theorem \ref{t1.Aimar}.

  Indeed, $(x,y)\in W_(r)$ is equivalent to $d(x,y)<r.$ By the definition of $d$, there exists $0<r'<r$ such that $(x,y)\in V(r')$. By (jjj),
  $V(r')\sse V(r),$ showing that $(x,y)\in V(r)$, i.e. $W_d(r)\sse V(r).$

  If $(x,y)\in V(\gamma r), $ then
  $$
  d(x,y)\le\gamma r<r\,,$$
  showing that $(x,y)\in W_d(r)$, i.e. $V(\gamma r)\sse W_d(r).$
\end{proof}

\subsection{Strong  b-metric spaces and completion} Let $(X,d)$ be a b-metric space. As we have seen, the topology $\tau_d$ generated by the b-metric $d$ has  some drawbacks  in what concerns the continuity property of $d$ and the topological openness of the ``open" balls. To remedy these shortcomings Kirk and Shahzad \cite[\S 12.4]{Kirk-Shah} introduced a special class of b-metrics. A mapping $d:X\times X\to [0,\infty)$ is called a \emph{strong b-metric} if it satisfies the conditions (i) and (ii) from \eqref{def.b-metric} and
\begin{equation}\label{s-b-metric}\tag{{\rm v}}
d(x,y)\le d(x,z)+sd(y,z)\,,
\end{equation}
for some $s\ge 1$ and all $x,y,z\in X.$ It is obvious that \eqref{s-b-metric} is equivalent to
\begin{equation}\label{s-b-m2}\tag{{\rm v$'$}}
d(x,y)\le \min\{sd(x,z)+d(y,z),d(x,z)+sd(y,z)\}\,,
\end{equation}
for all $x,y,z\in X, $ and that \eqref{s-b-metric} implies the $s$-relaxed triangle inequality (iii) from \eqref{def.b-metric}.

The topology generated by a strong b-metric has good properties as, for instance, the openness of the balls $B(x,r)$. Indeed, if $y\in B(x,r)$ then
$$d(y,z)\le d(x,y+sd(y,z)<\epsic\,,$$
provided $sd(y,z)<\epsic- d(x,y)$, that is $B(y,r')\sse B(x,r), $ where $r'= (\epsic- d(x,y))/s.$

Also   the following inequality
\begin{equation}\label{eq1.sbm}
|d(x,y)-d(x',y')|\le s[d(x,x')+d(y,y')]\,,
\end{equation}
holds for all $x,y,x',y'\in X,$  implying the continuity of the b-metric:  if $d(x_n,x)\to 0$ and $d(y_n,y)\to 0$, then the relations
$$
|d(x_n,y_n)-d(x,y)|\le s[d(x_n,x)+d(y_n,y)]\longrightarrow 0\;\mbox{ as }\; n\to\infty\,,
$$
show that $d(x_n,y_n)\lra d(x,y)$ as $n\to\infty.$

  A strong $b$-metric satisfies the $s$-polygonal inequality. Indeed,
\begin{align*}
  d(x_0,x_n)&\le  sd(x_0,x_1)+d(x_1,x_n)\le sd(x_0,x_1)+s d(x_1,x_2)+d(x_2,x_n)\le\dots\\
  &\le s[sd(x_0,x_1)+s d(x_1,x_2)+\dots+d(x_{n-1},x_n)]\,.
\end{align*}
\newpage

\textbf{Completeness and completion.}

A \emph{Cauchy sequence} in a b-metric space $(X,d) $ is a sequence $(x_n)$ in $X$ such that $\lim_{m,n\to\infty}d(x_n,x_m)=0.$ The inequality $d(x_n,x_m)\le s\left[d(x_n,x)+d(x,x_m)\right] $ shows that every convergent sequence is Cauchy. The b-metric space $(X,d)$ is called
\emph{complete} if every Cauchy sequence converges to some $x\in X$.
By a \emph{completion} of a b-metric space $(X,d)$ one understands   a complete b-metric space $(Y,\rho)$  such that there exists  an isometric embedding  $\,j:X\to Y$ with $j(X)$   dense in $Y$.

By an \emph{isometric embedding} of a  b-metric space $(X_1,d_1)$ into a b-metric space  $(X_2,d_2)$ one understands a mapping $f:X_1\to X_2$ such that $$d_2(f(x),f(y))=d_1(x,y)\,,$$ for all $x,y\in X_1.$ Two b-metric spaces $(X_1,d_1)$,  $(X_2,d_2)$ are called \emph{isometric} if there exists a surjective isometric embedding  $f:X_1\to X_2.$

The completeness is  preserved by the uniform  equivalence of b-metrics, but not by the  topological equivalence.

A question raised in \cite[p. 128]{Kirk-Shah} is: \smallskip

\emph{Does every strong b-metric space admit a completion}?\smallskip

This question was answered in the affirmative in \cite{an-dung16}.
\begin{theo}\label{t.compl-bm} Let  $(X,d)$ be a strong b-metric space.
 \begin{enumerate}\item[ \rm 1. ]There exists    a complete strong b-metric space   $(\tilde X,\tilde d\,)$ which  is a completion of $(X,d)$.
 \item[\rm 2. ] The completion is unique up to an isometry, in the sense that if  $(X_1,d_1)$, $(X_2,d_2)$ are two strong b-metric spaces which are completions of
 $(X,d)$, then $(X_1,d_1)$ and  $(X_2,d_2)$ are isometric.
 \end{enumerate}
\end{theo}\begin{proof}
The proof follows the ideas from the metric case. On the family $\rd C(X)$ of all Cauchy sequences in $X$ one considers the equivalence relation
$$
(x_n)\sim (y_n)\iff \lim_nd(x_n,y_n)=0\,.$$

On the quotient space $ \tilde X=\rd C(X)/\!\!\sim$ one defines $\tilde d$ by $\tilde d(\xi,\eta)=\lim_nd(x_n,y_n)$, where $(x_n)\in\xi$ and $(y_n)\in\eta,$
 and one shows  that $(\tilde X,\tilde d)$ is a complete strong b-metric space containing $X$ isometrically as a dense subset.
\end{proof}

\begin{remark}
As it is mentioned in  \cite{an-dung16}, the existence of a completion of an arbitrary b-metric space is still an  important open problem.
\end{remark}

\subsection{Spaces of homogeneous type}\label{Ss.homog-sp}

 Completing some earlier results of Coifman and de Guzman \cite{coif-guzman70}, Mac\'{\i}as and Segovia  \cite{maci-sego79a,maci-sego79b}  considered b-metrics (under the name quasi-distances)  in connection with some problems in    harmonic analysis.

The framework in \cite{maci-sego79a} is the following. Let $(X,d)$ be a b-metric space. One considers a positive measure $\mu$ defined on a $\sigma$-algebra of subsets of $X$ containing the open sets and the balls $B(x,r)$ such that
\begin{equation}\label{eq1.homog}
0<\mu\left(B(x,ar)\right)\le \beta \mu\left(B(x,r)\right)\,,
\end{equation}
for all $x\in X$ and $r>0$, where     $a>1$ and $A>0$ are fixed numbers. A b-metric space equipped with a measure $\mu$ satisfying \eqref{eq1.homog} is called a space of \emph{homogeneous type} and is denoted by $(X,d,\mu)$. If further, there exist $c_1,c_2>0$ such that
\begin{equation}\label{eq2.homog}
0<\mu(\{x\})<r<\mu(X)\;\Lra\; c_1 r\le \mu\left(B(x,r)\right)\le c_2r\,,
\end{equation}
for all $x\in X$, then the space $(X,d,\mu)$ of homogeneous type is called \emph{normal}.
\begin{remark}
  One can show that if $(X,d)$ is a b-metric space with a positive measure $\mu$ satisfying \eqref{eq2.homog}, then the space $(X,d,\mu)$ is of homogeneous type.
\end{remark}

     Concerning the openness of balls in b-metric spaces we mention the following result.
     \begin{theo}[\cite{maci-sego79a}] Let $(X,d)$ be a b-metric space. Then there exist a b-metric $d'$ on $X$, Lipschitz equivalent to $d$, and the constants $C>0$ and $0<\alpha<1$ such that
        \begin{equation}\label{eq1.maci}
        |d'(x,z)-d'(y,z)|\le Cr^{1-\alpha} d'(x,y)^\alpha,\end{equation}
         whenever $\max\{d'(x,z), d'(y,z)\}<r.$
            \end{theo}
  \begin{remark}
             The inequality \eqref{eq1.maci} can be written in the equivalent form
            \begin{equation}\label{eq2.maci}
        |d'(x,z)-d'(y,z)|\le C  d'(x,y)^\alpha\left(\max\{d'(x,z),d'(y,z)\}\right)^{1-\alpha},\end{equation}
         and it is easy to check that the balls corresponding to a b-metric $d'$ satisfying \eqref{eq2.maci} are $\tau_{d'}$-open.
   \end{remark}

  Indeed, let $B_{d'}(x,r)$ be a ball.   We have to show that for every $y\in B_{d'}(x,r)$ there exists $r'>0$ such that $B_{d'}(y,r')\sse B_{d'}(x,r),$ that is,
   $$
   d'(y,z)<r' \;\Lra \; d'(x,z)<r\,,$$
   for any $z\in X$. Supposing that $d'$ satisfies the $s'$-relaxed triangle inequality for some $s'\ge 1$, choose first $0<r'<r$. Then
   $$
   d'(x,z)\le s'd'(x,y)+s'd'(y,z)<2s'r\,.$$

   By \eqref{eq2.maci},
\begin{align*}
   d'(x,z)&\le d'(x,y)+|d'(x,z)-d'(x,y)|\\
   &<d'(x,y)+Cd'(y,z)^\alpha\left(\max\{d'(x,z),d'(x,y)\}\right)^{1-\alpha}\\
    &<d'(x,y)+C(2s'r)^{1-\alpha}(r')^\alpha\,.
\end{align*}

Choosing $0<r'<r\,$ such that $\,C(2s'r)^{1-\alpha}(r')^\alpha<r-d'(x,y)$, it follows $d'(x,z)<r.$

Concerning the set of points $x\in X$ with $\mu(\{x\})>0$ we mention.
\begin{prop}[\cite{maci-sego79a}] Let $(X,d,\mu)$ be a space of homogeneous type and
$$
M=\{x\in X :\mu(\{x\})>0\}\,.$$

Then the set $M$ is at most countable and for every $x\in M$ there exists $r>0$ such that $M\cap B(x,r)=\{x\}.$
\end{prop}

We mention also the following result.
\begin{theo}[\cite{maci-sego79a}] Let $(X,d,\mu)$ be a space of homogeneous type such that the balls are $\tau_d$-open. Let $\delta:X\times X\to[0,\infty)$ be given by
$$
\delta(x,y)=\inf\{\mu(B) : B\mbx{ is a ball containing } x,y\} \,,$$
if $x\ne y$ and $\delta(x,x)=0$.

Then $\delta$ is a b-metric on $X$, $(X,\delta,\mu)$ is a normal space and $\tau_\delta=\tau_d.$
\end{theo}

\subsection{Topological properties of $f$-quasimetric spaces} We present now, following \cite{arutyun17a} a very general class of  metric type spaces.
On a nonempty set $X$ consider a mapping $d:X\times X\to\Real_+$ satisfying only the condition
\bequ\label{def.dist}
d(x,y)=0\iff x=y\,,
\eequ
for all $x,y\in X$, and call such a function \emph{distance}. One can define open balls with respect to $d$ as usual
$$
B(x,r)=\{y\in X : d(x,y)<r\}\,,
$$
for $x\in X$ and $r>0$, and a topology $\tau_d$ by
$$
G\in\tau_d\overset{{\rm def}}{\iff} \forall x\in G,\;\exists r>0,\; B(x,r)\sse G\,,$$
where $G\sse X$.

The topology $\tau_d$ is $T_1$ because $X\sms\{x\}$ is open, and so $\{x\}$ closed. Indeed, if $y\ne x$, then $r:=d(y,x)>0$ and $x\notin B(y,r)$.

Along with the distance $d$ one can consider the \emph{conjugate distance} $\bar d(x,y)=d(y,x),\, x,y\in X$. The $\bar d$-balls are given by
$$
B_{\bar d}(x,r)=\{y\in X : \bar d(x,y)<r\}= \{y\in X :   d(y,x)<r\}\,,$$
and the corresponding topology is denoted by $\tau_{\bar d}.$

Consider now a function $f:\Real_+\times\Real_+\to\Real_+$ such that
\bequ\label{def.f}
 (t_1,t_2)\to  (0,0)\;\Lra \;  f(t_1,t_2)\to  (0,0)\,,
 \eequ
where  $(t_1,t_2)\in \Real_+\times\Real_+.$ we say that a distance $d$ on a set $X$ is an $f$-\emph{quasimetric} if it satisfies the inequality
 \bequ\label{ineq.f-tr}
 d(x,y)\le f(d(x,z),d(y,z))\,,
 \eequ
 for all $x,y,z\in X$.

\begin{example}\label{ex.f-fcs}
  We present first some important particular cases of function $f$.
 \bite
\item $f(t_1,t_2)=t_1+t_2$.  In this case $d$ is a quasimetric (see \cite{Cobzas}) and a metric if the distance $d$ is   symmetric.
\item $f(t_1,t_2)=s_1t_1+s_2t_2$, for some $s_1,s_2\ge 1.$ In this case $d$ is called an $(s_1,s_2)$-quasimetric (see \cite{arutyun17b}), a b-quasimetric if $s_1=s_2=s$,
respectively an $(s_1,s_2)$-metric, and b-metric if $d$ is symmetric.
\eite
\end{example}

From \ref{ineq.f-tr} one obtains the following result, called the \emph{asymptotic triangle inequality}:
\bequ\label{ineq.as-tr}
d(x_n,y_n)\to 0\mbox{ and } d(y_n,z_n)\to 0\;\Lra\; d(x_n,z_n)\to 0\,.\eequ

Conversely, if a distance functions satisfies \eqref{ineq.as-tr}, then there exists a function $f$, satisfying \eqref{def.f}, such that $d$ is an $f$-quasimetric.

Indeed, define $h:\Real_+\to\Real_+$  by
\beqs
h(t)=\sup\{d(u,v) : u,v\in X,\, \exists w\in X,\, d(u,w)+d(w,v)\le t\}\,.
\eeqs

The function  $h$ is obviously  nondecreasing and
$$\lim_{t\searrow 0}h(t)=0\,.$$

Indeed, if $t_n\to 0$, where  $t_n\in\Real_+,\, n\in\Nat,$ then there exist $u_n,v_n,w_n\in X$ such that
\begin{align*}
  d(u_n,w_n)+d(w_n,v_n)\le t_n\;\;\mbox{ and }\;\; d(u_n,v_n)>f(t_n)-\frac 1n\,,
\end{align*}
for all $n\in \Nat.$ The first inequality implies $d(u_n,w_n),d(w_n,v_n)\to 0,$ so that, by \eqref{ineq.as-tr},  $d(u_n,v_n) \to 0,$ which, by the second inequality from above, yields $f(t_n)\to 0.$

By the definition of the function $h$,
$$h(d(x,z)+d(z,y))\ge d(x,y)\,,$$
for all $x,y,z\in X$, so we can take $f(t_1,t_2)=h(t_1+t_2).$

Define the convergence of a sequence $(x_n)$ in a distance space $(X,d)$ to $x\in X$ by
$$
x_n\xrightarrow{d}x\overset{{\rm def}}{\iff} d(x,x_n)\to 0\,,
$$
and, for $Z\sse X$ and $x\in X$ put
$$
d(x,Z)=\inf\{d(x,z) : z\in Z\}\,.$$

We have the following useful characterizations of the interior and closure.
\begin{prop}\label{p1.Arutyun}
Let $(X,d)$ be an $f$-quasimetric space and $Z\sse X$. Then
\bequ\label{eq1.Arutyun}
  \inter (Z)=\{x\in X : d(x,X\sms Z)>0\}\;\mbox{ and }\; \clos(Z)=\{x\in X : d(x,A)=0\}\,.
  \eequ
\end{prop}\begin{proof} Let
$$
\tilde Z:=\{x\in X : d(x,X\sms Z)>0\}\,.$$

We show that
\begin{align*}
  &{\rm (i)}\quad \tilde Z\;\mbox{is open};\\
  &{\rm (ii)}\quad \tilde Z\sse Z;\\
   &{\rm (iii)}\quad \inter (Z)\sse\tilde Z\,,
   \end{align*}
which will imply that    $\inter (Z)=\tilde Z$.

Suppose that $\tilde Z$ is not open. Then there exists $x\in \tilde Z$ such that $B(x,n^{-1})\nsubseteq \tilde Z$ for all $n\in\Nat.$ Hence, for each $n\in\Nat,$ there exists
$y_n\in X$ such that
$$
d(x,y_n)<\frac1n\;\mbx{ and} \; d(y_n,X\sms Z)=0\,.$$
Then, for every $n\in\Nat,$ there exists
$w_n\in X\sms  Z$ such that $d(y_n,w_n)<1/n\,.$   By \eqref{ineq.as-tr}, $d(x,w_n)\to 0$, implying  $d(x,X\sms Z)=0$, in contradiction to the hypothesis that $x\in\tilde Z$.

The proof of (ii) is simple. If $x\notin Z$, then $x\in X\sms Z$, so that $d(x,X\sms Z)=0,$ that is, $x\notin\tilde Z$.

To prove (iii), suppose that $x\in\inter(Z)$. Then there exists $r>0$ such that $B(x,r)\sse\inter (Z)\sse Z.$ But then, for any $y\in X\sms Z,\, d(x,y)\ge r>0,$ that is,
$x\in \tilde Z$.

The proof of the formula for closure is based on the equality
$$
\clos(X\sms Y)=X\sms\inter(Y)\,,$$
valid for any subset $Y$ of $X$. Then
\begin{align*}
 \clos(Z)&=X\sms\inter(X\sms Z)
 =X\sms \{x\in X : d(x,Z)>0\}\\&=  \{x\in X : d(x,Z)=0\}\,.
\end{align*}
 \end{proof}

 Proposition \ref{p1.Arutyun} has some important consequences.
 \begin{corol}\label{c1.Arutyun} Let $(X,d)$ be an $f$-quasimetric space.
 \ben
 \item[\rm 1.] For every $x\in X$ and $r>0,\, x\in \inter\big(B(x,r)\big)$, or, equivalently, $B(x,r)$ is a neighborhood of $x$.
 \item[\rm 2.] The topology $\tau_d$ satisfies the first axiom of countability, i.e. every point has a countable base of neighborhood. Consequently, usual sequences suffices to characterize the topological properties of $X$.
 \item[\rm 3.]  The convergence of a sequence $(x_n)$ in $X$ to $x\in X$ with respect to $\tau_d$ is characterized in the following way:
 \bequ\label{eq1.d-converg}
 x_n\xrightarrow{\tau_d}x\iff   d(x,x_n)\lra 0\,.\eequ
\een\end{corol}\begin{proof}
  1. This follows  from the following relations
  \begin{align*}
    y\in X\sms B(x,r)&\iff d(x,y)\ge r\\
    &\;\,\Lra \, d(x,X\sms B(x,r))\ge r>0\\
    &\iff x\in\inter \big(B(x,r)\big)\,.
  \end{align*}

  2. A countable base of neighborhoods of a point $x\in X$ is $B(x,n^{-1}),\, n\in\Nat.$  If $V$ is a neighborhood of $x$, then there exists   $G\in\tau_d$ such that
  $$ x\in G\sse V\,.$$

  By the definition of the topology $\tau_d$, there exists $r>0$ such that $ B(x,r)\sse G,$ implying $$B(x,n^{-1})\sse B(x,r)\sse G\sse V\,,$$
  for some sufficiently large $n\in\Nat.$

  3. Suppose that $x_n\xrightarrow{\tau_d}x $  and let $r>0$. Then, by 1,  $B(x,r)\in\rd V(x)$ so there exists   $n_r\in\Nat$ such that $x_n\in B(x,r)$, or, equivalently, $d(x,x_n)<r$ for all $n\ge n_r.$ Consequently, $d(x,x_n)\lra  0.$

  Suppose now that $d(x,x_n)\lra 0$ and let $V\in\rd V(x).$ Then $x\in\inter(V)\in\tau_d$, so there exists $r>0$ such that $ B(x,r)\sse\inter(V)\sse V.$ By hypothesis, there exists
  $n_0\in\Nat$, such that $d(x,x_n)<r$ for all $n\ge n_0$, implying $x_n\in B(x,r)\sse V$ for all $n\ge n_0.$
\end{proof}

\begin{remark} The convergence with respect to the topology $\tau_{\bar d}$ generated by the conjugate $f$-quasimetric $\bar d(x,y)=d(y,x),\, x,y\in X,$ is characterized by
\bequ\label{eq2.d-converg}
 x_n\xrightarrow{\tau_{\bar d}}x\iff d(x_n,x)\lra 0\,.\eequ
\end{remark}

Since $x\in \inter\big(B(x,r)\big),$ there exists $r'>0$ such that $B(x,r')\sse\inter\big(B(x,r)\big)\sse B(x,r).$ A natural question is to find the biggest $r'$ such that
$B(x,r')\inter\big(B(x,r)\big).$

For $r>0$ let
$$
\Lbda(r)=\{t_1\ge 0 :\limsup_{t_2\searrow 0}f(t_1,t_2)\ge r\}\,,$$ and
$$
\theta(r)=\begin{cases}
  \sup\Lbda(r) &\mbox{ if } \; \Lbda(r)\ne \emptyset \\
  r &\mbox{ if } \; \Lbda(r)= \emptyset\,.
\end{cases} $$

Here, by definition,
\bequ\label{limsup} \limsup_{t_2\searrow 0}f(t_1,t_2)=\inf_{\delta>0}\big(\sup\{f(t_1,t_2) : 0\le t_2<\delta\}\big)\,.\eequ

\begin{remark}\label{re1.Arutyun}
  Observe that $\Lbda(r)=\emptyset$ implies $B(x,r)=X$.

  Also $\theta(r)>0$ in both cases.
\end{remark}

Indeed, if $\Lbda(r)=\emptyset,$ then putting $t_1=d(x,y)$ for some arbitrary $y\in X$, it follows $\limsup_{t_2\searrow 0}f(t_1,t_2)<r$, so there exists $\delta >0$ such that
$\sup\{f(t_1,t_2) : 0\le t_2<\delta\}<r$, implying
$$
d(x,y)\le f(d(x,y),d(y,y))=f(t_1,0)<r\,,$$
that is, $y\in B(x,r).$

If $\theta(r)=0$, then  $\Lbda(r)\ne\emptyset,$ so there exists a sequence $(t_1^k)$ in $\Lbda(r)$ such that $t^k_1\to 0$ as $k\to\infty.$
By \eqref{limsup} and the definition of $\Lbda(r)$, for every $k\in\Nat$ there exists $0\le t^k_2<1/k$ such that $f(t^k_1,t^k_2)>r/2.$
Taking into account \eqref{def.f}, one obtains the contradiction.
$$
0=\lim_{k\to\infty} f(t^k_1,t^k_2)\ge r/2\,.$$

\begin{prop}\label{p2.Arutyun}  Let $(X,d)$ be an $f$-quasimetric space.
Then
$$
B(x,\theta(r))\sse\inter\big(B(x,r)\big)\,,$$
for every $x\in X$ and $r>0$.
 \end{prop}\begin{proof} By Remark \ref{re1.Arutyun} we can suppose $\Lbda(r)\ne\emptyset.$

Let $y\in B(x,\theta(r)).$ Then   $t_1:=d(x,y)<\theta(r),$ implying $t_1\notin\Lbda(r)$ so that $\limsup_{t_2\searrow 0}f(t_1,t_2)<r.$ By \eqref{limsup}
 there exists $\delta>0$ such that
 \bequ\label{eq3a.Arutyun}\sup\{f(t_1,t_2) : 0\le t_2<\delta\}<r\,.\eequ

 We show that \bequ\label{eq3b.Arutyun} B(y,\delta)\sse B(x,r)\,.\eequ

 If $z\in B(y,\delta), $ then, by \eqref{eq3a.Arutyun}
 $$
 d(x,z)\le f(d(x,y),d(y,z))=f(t_1,d(y,z))<r\,,$$
 because $d(y,z)<\delta.$

 Since  $B(y,\delta)$ is a neighborhood of $y$, the inclusion \eqref{eq3b.Arutyun} shows that $B(x,r)$ is also a neighborhood of $y$, that is, $y\in\inter\big(B(x,r)\big).$
   \end{proof}

   \begin{remark} If $(X,d)$ is an $(s_1,s_2)$-quasimetric space, i.e. an $f$-quasimetric space for $f(t_1,t_2)=s_1t_1+s_2t_2$, then
   $$
   \theta (r)=r/s_1\,.$$
   \end{remark}

   Indeed,
\begin{align*}
   \theta(r)&=\inf\{t_1\ge 0 : \lim_{t_2\searrow 0}(s_1t_1+s_2t_2)\ge r\}\\
  &=\inf\{t_1\ge 0 : s_1t_1\ge r\}=r/s_1\,.
\end{align*}

The authors define in \cite{arutyun17a} the notion  of Cauchy sequence and completeness. A sequence in a distance space $(X,d)$ is called \emph{Cauchy} if for every $\epsic >0$ there exists $n_0=n_0(\epsic)$ such that
\bequ\label{Cauchy-l}
d(x_n,x_{n+k})<\epsic\,,
\eequ
for all $n\ge n_0$ and all $k\in\Nat.$ The distance space $(X,d)$ is called \emph{complete} if every Cauchy sequence is convergent to some $x\in X$.

The authors prove in \cite{arutyun17a} the validity of Baire category theorem in complete $f$-quasimetric spaces which satisfy the separation axiom $T_3$  (i.e.  are regular). As in the metric case, the proof is based on the nonemptiness of descending sequences  of closed balls with radii tending to 0. They extend  the metrization Theorems \ref{t.Frink} and \ref{t.Stemp} to this setting proving the quasimetrizability of $f$-quasimetric spaces. In this case, the equivalent quasimetrics are also given by the formulae \eqref{def.1-metric} and \eqref{def.p-metric} and are denoted by Inf $d$. They introduce the notion of weak symmetry of the $f$-quasimetric $d$ by the condition
\bequ\label{eq.w-sym}
d(x,x_n)\to 0\;\Lra\; d(x_n,x)\to 0\,,\eequ
for all sequences $(x_n)$ in $X$ and $x\in X$. The topology  generated by a weakly symmetric $f$-quasimetric is normal and
 metrizable.
\begin{remark}
 By \eqref{eq1.d-converg} and  \eqref{eq2.d-converg}, the condition  \eqref{eq.w-sym} means that the identity mapping $I:(X,\tau_d)\to (X,\tau_{\bar d})$ is continuous, or equivalently, $\tau_{\bar d}\sse\tau_d,$ i.e. the topology $\tau_{d}$ is finer than $\tau_{\bar d}$.
\end{remark}

\begin{remark} In the theory of quasimetric spaces (see Example \ref{ex.f-fcs}) a sequence satisfying \eqref{Cauchy-l} is called ``left $K$-Cauchy"   and the corresponding notion of completeness, ``left $K$-completeness". If the  sequence $(x_n)$ satisfies the condition
\bequ\label{Cauchy-r}
d(x_{n+k},x_{n})<\epsic\,,
\eequ
for all $n\ge n_0$ and all $k\in\Nat,$ then it is called ``right $K$-Cauchy",   and the corresponding notion of completeness, ``right  $K$-completeness".

Some authors call   a sequence $(x_n)$ satisfying \eqref{Cauchy-l} ``forward Cauchy" and ``backward Cauchy" if it satisfies \eqref{Cauchy-r}. Also  the convergence given by $d(x,x_n)\to 0$ is called  ``backward convergence", while that given by   $d(x_n,x)\to 0$ is called  ``forward convergence" (see, e.g. \cite{menuci13}). Combining these notions of  Cauchy sequence and convergence one obtains various notions of completeness: ``forward-forward complete" meaning that every forward Cauchy sequence is forward convergent, with similar definitions for forward-backward,   backward-forward, etc  -- completeness.

Due to the asymmetry of the quasimetric,  there are several  notions of Cauchy sequence (actually 7, see \cite{reily-subram82}), each of them agreeing with the usual notion of Cauchy (fundamental) sequence in the metric case. Considering $d$-convergence and $\bar d$-convergence, from  these 7 notions of Cauchy sequence one obtains 14  notions of completeness (see the  book \cite{Cobzas}).
\end{remark}

 \subsection{Historical remarks and further results}  The relaxed triangle inequality and the corresponding spaces were rediscovered several times under various names -- quasi-metric, near metric (in \cite{Deza}),  metric type, etc.
 \bite
\item  (1970)   Coifman and de Guzman \cite{coif-guzman70} in connection with some problems in harmonic analysis (a b-metric is called ``distance" function);
 \item (1979) the results of Coifman and de Guzman  were completed by Mac\'{\i}as and Segovia  \cite{maci-sego79a,maci-sego79b};
  \item (1989) Bakhtin \cite{bahtin89} called them ``quasi-metric spaces" and proved a contraction principle for such spaces;
   \item (1993) Czerwik introduced them under the name ``b-metric space", first for $s=2$ in \cite{czerw93}, and then for an arbitrary $s\ge 1$ in \cite{czerw98}, with applications to fixed points;
       \item (1998,2003) Fagin \emph{et al.} \cite{fagin03,fagin98} considered distances satisfying the $s$-relaxed triangle and polygonal inequalities with applications to some problems in theoretical computer science;
\item (2010)  Khamsi \cite{khamsi10a} introduced them under the name ``metric type spaces" and  remarked that if $D$ is a cone metric on a set $X$ with values in a Banach space ordered by a normal cone $C$ with normality constant $K$, then $d(x,y)=\|D(x,y)\|,\, x,y\in X,$ is a b-metric on $X$ satisfying the $K$-relaxed polygonal inequality.
  \eite

     Some topological properties  of  b-metric spaces (e.g. compactness) were studied in \cite{khamsi10b}. Xia \cite{xia09} studied the properties of the space $C(T,X)$ of continuous functions from a compact metric space $T$ to a b-metric space $X$, and  geodesics and intrinsic metrics in b-metric  spaces. The results were applied to show that   the optimal transport paths between atomic probability measures are  geodesics in the intrinsic metric. An, Tuyen and Dung \cite{an-dung15} extended to b-metric spaces Stone's paracompactness theorem.

 \section{Generalized b-metric spaces}

The notions of \emph{generalized metric}, meaning  a mapping $d:X\times X\to [0,\infty]$ satisfying the axioms of a metric, and generalized metric space $(X,d)$  were introduced by W. A. J. Luxemburg in \cite{lux1}--\cite{lux3} in connection with the method of successive approximation and fixed points. These results   were completed by A. F. Monna \cite{monna61} and M. Edelstein \cite{edelst64}.  Further results were obtained by  J. B. Diaz and B. Margolis \cite{diaz-margol68,margolis68} and C. F. K.  Jung \cite{jung69}.    G. Dezs\H{o} \cite{dezso00}  considered  generalized vector metrics, i.e. metrics with values in $\Real^m_+\cup\{(+\infty)^m\}$, and extended to this setting  Perov's fixed point theorem  (see \cite{perov64} -- \cite{perov66}) as well as other fixed point results (Luxemburg, Jung, Diaz-Margolis,  Kannan).

For some recent results on generalized metric spaces see \cite{beer13} and \cite{czerw-krol16b}.
Recently, G. Beer and J. Vanderwerf \cite{beer15}--\cite{beer-vdw15a}  considered  vector spaces equipped with norms that  can take infinite values, called   ``extended norms" (see also \cite{czerw-krol16a}).

Following these ideas, we consider here the notion of \emph{generalized} b-\emph{metric}  on a nonempty set $X$ as a mapping $d:X\times X\to [0,\infty]$ satisfying the conditions (i)--(iii) from (\ref{def.b-metric}). If  $d$      satisfies further the condition  \eqref{s-b-metric}, then $d$
is called a \emph{generalized strong} b-\emph{metric} and the pair $(X,d)$ a \emph{generalized strong} b-\emph{metric space}.

Let $(X, d)$ be a generalized   b-metric space. As in Jung \cite{jung69}, it follows that
\begin{equation}\label{zad3}
x \sim y  \stackrel{d}{\Longleftrightarrow} d(x,y) < + \infty, \quad x,y \in X,
\end{equation}
is an equivalence relation on $X$. Denoting by $X_i, \, i \in I$, the equivalence classes corresponding to $\sim$ and putting $d_i = d |_{X_i \times X_i}, \, i \in I$, then  $(X_i, d_i)$ is a b-metric space (a strong b-metric space if $(X,d)$ is a generalized strong b-metric space) for every $i\in I$. Therefore, $X$ can be  uniquely decomposed into equivalence classes $X_i, \, i \in I$, called the \emph{canonical decomposition} of $X$.

By analogy to   \cite{jung69} we have.

\begin{theo}\label{t.2}
Let $(X,d)$ be a generalized  b-metric space and $X_i, \, i \in I,$ its   canonical decomposition. Then the following hold.

\begin{enumerate}
    \item[\rm 1. ] The space  $(X,d)$ is complete if and only if  $(X_i,d_i)$ is complete for every $i \in I$.
    \item[\rm 2. ] If  $\,(Y_i, d_i), \, i \in I$, are   b-metric spaces (with the same $s$) and $Y_i \cap Y_j = \emptyset  $ for all $i \neq j$  in $I$, then

\begin{equation}\label{zad4}
d(x,y) := \left\{ \begin{array}{lcl}
d_i(x,y) & \mbox{if} & x,y \in Y_i, \text{ for some } i \in I, \\
 +\infty & \mbox{if} & x \in Y_i \text{ and } y \in Y_j \\
& & \text{ for some } i,j \in I \text{ with } i \neq j, \\
\end{array}\right.
\end{equation}
is a generalized  b-metric on $Y = \bigcup_{i \in I} Y_i$, with $\{Y_i : i\in I\}$   the family of equivalence classes corresponding to the equivalence relation (\ref{zad3}).
\end{enumerate}

The same results are true for generalized  strong  b-metric spaces.
\end{theo}
\subsection{The completion of generalized b-metric spaces} In this subsection we shall prove the existence of the completion of strong b-metric spaces. The existence of the completion of a generalized metric space was proved in \cite{czerw-krol16}.

We start with the following  lemma.

\begin{lemma}\label{le1}
Let  $(X,d)$ be a generalized  b-metric space, ($Z,D$) a complete generalized   b-metric space, with continuous generalized b-metrics $d,D$ and  $Y$ a dense subset of $X$. Then for every    isometric embedding $f: Y \rightarrow Z$  there exists a unique isometric embedding $F: X \rightarrow Z$ such that $F|_Y = f$. If, in addition, $X$ is  complete and $f(Y)$ is dense in $Z$, then $F$ is bijective (i.e. $F$ is an isometry of $X$ onto $Z$).
\end{lemma}\begin{proof}
 For the sake of completeness we include the simple proof of this result.
For $x\in X$ let $(y_n)$ be a sequence in $Y$ such that $d(y_n,x)\to 0.$ Then $(y_n)$ is a Cauchy sequence in $(X,d)$ and the equalities $D(f(y_n),f(y_m))=d(y_n,y_m),\, m,n\in\mathbb{N},$ show that $(f(y_n))$ is a Cauchy sequence in $(Z,D)$. Since $(Z,D)$ is complete, there exists $z\in Z$ such that $D(f(y_n),z)\to 0$. If $(y'_n)$ is another sequence in $Y$ converging to $x$, then $(f(y'_n)) $ will converge to an element $z'\in Z$. By the continuity of the generalized b-metrics $d$ and $D$,
$$D(z,z')=D(\lim_nf(y_n),\lim_nf(y'_n))= \lim_nD(f(y_n),f(y'_n))= \lim_nd(y_n,y'_n)=0\,,$$
showing that $z=z'$. So we can unambiguously define a mapping $F:X\to Z$ by $F(x)=\lim_nf(y_n)$, where $(y_n)$ is a sequence in  $Y$ converging to $x\in X$.
For $y\in Y$ taking $y_n=y,\, n\in\mathbb{N}$, it follows $F(y)=y$.

For $x,x'\in X$, let $(y_n), (y'_n)$ be sequences in $Y$ converging to $x$ and $x'$, respectively. Then
$$D(F(x),F(x'))=\lim_nD(f(y_n),f(y'_n))= \lim_nd(y_n,y'_n)=d(x,x')\,,$$
i.e. $F$ is an isometric embedding.

If $f(Y)$ is dense in $Z$, then, for any $z\in Z$, there exists a sequence $(y_n)$ in $Y$ such that $ D(f(y_n),z)\to 0. $ It follows that $(f(y_n))$ is a Cauchy sequence in $Z$ and so, as $f$ is an isometry,   $(y_n)$ will be a Cauchy sequence in $X$.  As the space $X$ is complete, $(y_n)$ is convergent to some $x\in X$. But then
$$
D(F(x),z)=\lim_nD(F(x),f(y_n))=\lim_n d(x,y_n)=0\,,$$
showing that $F(x)=z.$
\end{proof}

\begin{remark}
  The proof can be adapted to show that, under the hypotheses of Lemma \ref{le1}, every  uniformly continuous mapping $f:Y\to Z$ has a unique uniformly continuous extension to $X$. The notion of uniform continuity for mappings between generalized b-metric spaces is defined as in the metric case.
\end{remark}

Let $(X,d)$ be a generalized strong b-metric space with $X_i, \, i \in I$, the family of equivalence classes corresponding to (\ref{zad3}). For every $i \in I$, let $(Y_i, D_i)$ be a completion of the strong b-metric space  $(X_i,d_i)$. Denote by $T_i : (X_i, d_i) \rightarrow (Y_i, D_i)$ the isometric embedding with $T_i (X_i)$ $D_i$-dense in $Y_i$ corresponding to this completion.

Replacing, if necessary, $Y_i$ with $\overline{Y_i} = Y_i \times \{ i\}, D_i$ with $\overline{D_i} ((x,i),(y,i))=D_i(x,y),$ for $ x,y \in Y_i,$ and putting $\overline{T_i}(x,i) = (T_i(x), i),\, x\in Y_i$, we may suppose, without restricting the  generality, that
$$ Y_i \cap Y_j = \emptyset \text{ for all } i,j \in I \text{ with } i \neq j\,. $$

Put $Y := \bigcup_{i \in I} Y_i$, and define
$$ D: Y \times Y \rightarrow [0, \infty] $$

\noindent
according to (\ref{zad4}) and $T: X \rightarrow Y$ by
$$ T(x) := T_i(x), $$
where $i$ is the unique element of $I$ such that $x \in X_i$.

We have the following result.

\begin{theo}\label{theorem-3}
Let $(X,d)$ be a generalized strong b-metric space and $(Y, D)$ the generalized strong b-metric space defined above. Then

\begin{enumerate}
    \item[\rm (i) ] $(Y, D)$ is a complete generalized strong  b-metric space;
    \item[\rm (ii) ] $T: (X,d) \rightarrow (Y, D)$ is an isometric embedding with $T(X)$ $D$-dense in $Y$;
    \item[\rm  (iii) ] any other complete generalized  strong b-metric space $(Z,\varrho )$   that contains a $\rho$-dense isometric copy of $(X,d)$, is isometric to $(Y,D)$.
\end{enumerate}
\end{theo}\begin{proof}
Since each strong b-metric space $(Y_i, D_i)$ is complete, Theorem \ref{t.2} implies that the generalized strong b-metric space $(Y,D)$ is complete.

Let $x,y \in X$. If $x,y \in X_i$, for some $i \in I$, then

$$ D(T(x), T(y)) = D_i (T_i(x), T_i(y)) = d_i (x,y) = d(x, y). $$

If $x \in X_i,\, y \in X_j$ with $i \neq j$, then

$$ T(x) = T_i(x) \in Y_i \text{ and } T(y) = T_j(x) \in Y_j\,, $$
 so that

$$D(T(x), T(y)) = D(T_i(x), T_j(y)) = + \infty = d(x,y).$$ Now for $ \xi \in Y$ there exists a unique $i \in I$ such that $\xi \in Y_i$. Since $T_i(X_i)$ is dense in $(Y_i, D_i)$, there exists a sequence $(x_n)$ in $X_i$ such that
$$ 0 = \lim_{n \to \infty} D_i(T_i(x_n), \xi ) = \lim_{n \to \infty} D(T(x_n), \xi),
$$
which means that $T(X)$ is $D$-dense in $(Y,D)$.

Finally, to verify (iii), let $ S: (X,d) \rightarrow (Z,\rho)$ be an isometric embedding with $S(X)$ dense in $Z$. Define $R:T(X)\to X$ by $R(T(x))=x,\, x\in X$. Then $R$ is an isometry of $T(X)$ onto $X$ and $S\circ R$ is an isometric embedding of $T(X)$ into $Z$. Since $T(X)$ is dense  in $Y$ and $S(R(T(X)))=S(X)$ is dense in $Z$,
 Lemma \ref{le1} yields the existence of  an isometry $U$ of $Y$ onto $Z$, which ends the proof.\end{proof}

\section{Fixed points in b-metric spaces}
We shall prove some  fixed point results  in b-metric and in generalized b-metric spaces.
\subsection{Fixed points in b-metric spaces}

 We start with the case of b-metric spaces.
 The second result is an extension to b-metric spaces of Theorem 4.1 from \cite{Dugundji-G}.

  Let $(X,d)$ be a b-metric space with $d$ satisfying the $s$-relaxed triangle inequality.
  We consider functions $\vphi:\Real_+\to\Real_+$ satisfying the conditions
  \begin{equation}\label{def.phi}
  \begin{aligned}
    &{\rm (a)}\quad \vphi  \;\mbx{ is  nondecreasing  and}\\
  &{\rm (b)}\quad \lim_{n\to\infty}\vphi^n(t)=0\;\mbx{  for all }\; t>0\,.\\
   \end{aligned}\end{equation}

  \begin{remark}\label{re.phi} If $\vphi:\Real_+\to\Real_+$ satisfies the conditions (a) and (b) from above, then
  \begin{align*}
  {\rm(c) } \quad &\vphi(t)<t\;\mbx{ for all }\;t>0; \\
  {\rm(d) } \quad &\lim_{n\to\infty}\vphi^n(0)=0\;\mbx{ and }\; \vphi(0)=0=\lim_{t\searrow 0}\vphi(t)\,.
  \end{align*}
   \end{remark}

  Indeed, if $\vphi(t)\ge t$ for some $t>0$, then, by (a),  $\vphi^2(t)\ge \vphi(t)\ge t$ and, in general $\vphi^n(t)\ge t>0$ for all $n$, in contradiction to (b).
  
  Also, $0\le\vphi(0)\le\vphi(1)$ implies $0\le\vphi^2(0)\le \vphi^2(1)$ and in general  $0\le\vphi^n(0)\le\vphi^n(1)$. Since $\lim_{n\to\infty}\vphi^n(1)=0,$ this yields (d).
  
  Similarly,  $0\le\vphi(0)\le \vphi(t)<t$ for any $t>0$, implies $\vphi(0)=0=\lim_{t\searrow 0}\vphi(t)$.

\begin{theo}\label{t1.Czerw} Let $(X,d)$ be a complete b-metric space, where $d$ satisfies the $s$-relaxed triangle inequality and let  $\vphi:\Real_+\to\Real_+$  be a function satisfying the conditions (a), (b) from \eqref{def.phi}.
Then every mapping $f:X\to X$ satisfying the inequality
\begin{equation}\label{eq1.Cz}
d(f(x),f(y))\le\vphi(d(x,y))\,,
\end{equation}
for all $x,y\in X$, has a unique fixed point $z$ and, for every $x\in X$, the sequence $\big(f^n(x)\big)_{n\in\Nat_0}$ converges to $z$ as $n\to \infty$.
  \end{theo}\begin{proof} We present the proof given in \cite{kaj-luk18}. Let $x\in X$. Put $x_n=T^nx,\, n\in\Nat_0\,,$  and let us show that $(x_n)_{n\in\Nat_0}$ is a Cauchy sequence.
  
  Observe first that \eqref{eq1.Cz} implies
  \bequ\label{eq2.Cz}
  d(T^nu,T^nv)\le\vphi^n(d(u,v))\,,
  \eequ
  for all $u,v\in X$ and $n\in\Nat.$ This implies 
  \bequ\label{eq3.Cz}
  d(T^nx_{kn},x_{kn})\le\vphi^{kn}(d(x_n,x_0))\,,
    \eequ
    for all $k,n\in\Nat.$ Indeed, by \eqref{eq2.Cz},
    $$  d(T^nx_{kn},x_{kn})= d(T^{kn}x_{n},T^{kn}x_{0})\le \vphi^{kn}(d(x_nx_0))\,.$$
    
    From \eqref{eq3.Cz} and \eqref{def.phi}.(b) one obtains 
    \bequ\label{eq3b.Cz}
  \lim_{k\to\infty}d(T^nx_{kn},x_{kn})=0\,,
    \eequ
    for every $n\in\Nat.$ 
    
 Let $\epsic>0$ be  given.
 
 Observe first that  there exist $\bar k,\bar n\in\Nat$ s.t.
    \bequ\label{eq4.Cz}\begin{aligned}
      {\rm(i)}\quad &T^{\bar n}(B(x_{\bar k\bar n},\epsic)\sse B(x_{\bar k\bar n},\epsic);\\
      {\rm(ii)}\quad & x_{k\bar n}\in B(x_{\bar k\bar n},\epsic)\;\mbx{ for all }\; k\ge \bar k;\\
      {\rm(iii)}\quad & d(x_{k_1\bar n},x_{\bar k_2\bar n})<2s\epsic\;\mbx{ for all }\; k_1,k_2\ge \bar k.
    \end{aligned}\eequ
    
    Indeed, by \eqref{def.phi}.(b), there exists $\bar n\in\Nat$ s.t.
    \bequ\label{eq5.Cz}
    \vphi^n(\epsic)<\epsic/(2s)\;\mbx{ for all }\; n\ge\bar n,
    \eequ 
    and, by \eqref{eq3b.Cz}, there exists $\bar k\in\Nat$ s.t. 
    \bequ\label{eq6.Cz}
    d(T^{\bar n}x_{k\bar n},x_{k\bar n})<\epsic/(2s)\;\mbx{ for all }\; k\ge\bar k\,.
    \eequ
    
    But then, $d(u,x_{\bar k\bar n})<\epsic$ implies
    
    $$
    d(T^{\bar n}u,T^{\bar n}x_{\bar k\bar n})\le\vphi^{\bar n}(d(u,x_{\bar k\bar n}))\le   \vphi^{\bar n}(\epsic)<\epsic/(2s)\,,$$
    so that
    \begin{align*}      
    d(T^{\bar n}u,x_{\bar k\bar n})&\le s\left[ d(T^{\bar n}u,T^{\bar n}x_{\bar k\bar n})+d(T^{\bar n}x_{\bar k\bar n},x_{\bar k\bar n})\right]\\
    &<s\left[\frac\epsic{2s}+\frac\epsic{2s}\right]=\epsic\,,
    \end{align*}
    showing that \eqref{eq4.Cz}.(i) holds. 
    
    Since $x_{\bar k\bar n}\in B(x_{\bar k\bar n},\epsic)$ it follows that $x_{(\bar k+1)\bar n}=T^{\bar n}x_{\bar k\bar n} \in B(x_{\bar k\bar n},\epsic)$
    and, in general, by induction,  $ x_{(\bar k+j)\bar n}  \in B(x_{\bar k\bar n},\epsic)$ for any $j\in\Nat_0.$
    
     Now, by (ii), $x_{k_1\bar n},x_{k_2\bar n}\in B(x_{\bar k\bar n},\epsic)$ for all $k_1,k_2\ge \bar k,$ so that
    $$
    d(x_{k_1\bar n},x_{k_2\bar n})\le s( d(x_{k_1\bar n},x_{\bar k\bar n})+d(x_{\bar k\bar n},x_{k_2\bar n})<2s\epsic\,,$$
    showing that (iii) holds too.   
    
    By \eqref{eq3b.Cz} for $n=1$ one obtains
    \bequs
    \lim_{k\to\infty}d(x_{k+1},x_k)=0\,.\eequs
    
   It is easy to check that  this implies 
 \bequs
      \lim_{k\to\infty}d(x_{k\bar n+p},x_{k\bar n})=0 \;\mbx{ for }\; p=0,1,\dots,\bar n-1\,,
      \eequs
      so there exists $k_0\in \Nat$ s.t. 
      \bequ\label{eq8.Cz}
       d(x_{k\bar n+p},x_{k\bar n})<\epsic\;\mbx{ for all  }\; k\ge k_0 \;\mbx{ and }\;  p=0,1,\dots,\bar n-1\,.
       \eequ

     Let now $\tilde k:=\max\{\bar k, k_0\}$  and let  $m_1=k_1\bar n+p_1,\, m_2=k_2\bar n+p_2$ with
       $p_1,p_2\in\{0,1,\dots,\bar n-1\}$ and $k_1,k_2\ge \tilde k$.

     Combining \eqref{eq8.Cz} and \eqref{eq4.Cz}.(iii) one obtains
     \begin{align*}
       d(x_{m_1},x_{m_2})&\le sd(x_{k_1\bar n+p_1},x_{k_1\bar n})+s^2d(x_{k_1\bar n},x_{k_2\bar n})+s^3d(x_{ k_2\bar n},x_{k_2\bar n+p_2})\\
       &<(s+2s^3+s^2)\epsic\le 4s^3 \epsic\,,
     \end{align*} 
     which shows that $(x_n)$ is a Cauchy sequence. 
     
     The completeness of $X$ implies the existence of a point $z\in X$ s.t. $\lim_{n\to\infty}d(x_n,z)=0.$
     
     We have 
     $$
     d(x_{n+1},Tz)=d(Tx_n,Tz)\le\vphi(d(x_n,z))\le d(x_n,z)$$
     for all $n\in\Nat,$ so that

     \begin{align*}
       d(z,Tz)&\le s\left[d(z,x_{n+1})+d(x_{n+1},Tz)\right]\\ 
       &\le   s\left[d(z,x_{n+1})+d(x_{n},z)\right]\,.
     \end{align*}
     
     Letting $n\to \infty$ one obtains $d(z,Tz)=0,$ that is, $z=Tz.$
     
     The uniqueness follows in the following way. Suppose $z_i=Tz_i,\, i=1,2.$ Then
     $$
     d(z_1,z_2)=d(Tz_1,Tz_2)\le\vphi(d(z_1,z_2))\,.$$
     
     By Remark \ref{re.phi}.(c) this can hold only for $d(z_1,z_2)=0,$ that is, for $z_1=z_2.$   
  \end{proof}

Let $(X,d)$ be a b-metric space with $d$ satisfying the s-relaxed triangle inequality for some $s\ge 1.$ An important  particular case of a function $\vphi$ satisfying the conditions (a),(b) from \eqref{def.phi} is
$$
\vphi(t)=\alpha t,\; t\ge 0\,.$$

If  $0<\alpha <1,$  then
 
$$
\vphi^n(t)=\alpha^n t\to 0 \; \mbx{ as }\; n\to \infty\,.$$
 Since $\vphi$ is also strictly increasing, it satisfies the conditions (a),(b) from \eqref{def.phi}.

The inequality  \eqref{eq1.Cz} becomes in this case
\begin{equation*}
d(f(x),f(y))\le \alpha d(x,y)\,,
\end{equation*}
for all $x,y\in X$.

So,  Theorem \ref{t1.Czerw} has as consequence the analog of Banach contraction principle in b-metric spaces. The following proposition  also  illustrates how
various types of relaxed  inequalities for the b-metric influence the form this principle takes.

\begin{prop}\label{c1.Bahtin}  Let $(X,d)$ be a complete b-metric space, where $d$ satisfies the $s$-relaxed triangle inequality and
$f:X\to X$ a mapping such that, for some $0<\alpha <1,$
\begin{equation}\label{eq1.Bahtin}
d(f(x),f(y))\le\alpha d(x,y),
\end{equation}
for all $  x,y\in X.$   Then $f$
 has a unique fixed point $z$ and, for every $x\in X$, the sequence $\big(f^n(x)\big)_{n\in\mathbb{N}}$ converges to $z$ as $n\to \infty$.

\begin{enumerate}
\item[\rm 1. (\cite{bahtin89})]   
If further  $0<\alpha<1/s,$ then the following evaluation of the order of convergence holds
\begin{equation}\label{ineq2.Bahtin}
d(x_{n},z)\le\frac{s^2d(x_0,x_1)}{1-\alpha s}\cdot \alpha^{n}\,,
\end{equation}
for all $n\in\mathbb{N}.$
\item[\rm 2. (\cite{Kirk-Shah})] If $d$ satisfies the $s$-relaxed polygonal inequality, then  the following evaluation of the order of convergence
    \begin{equation}\label{ineq2.KS2}
d(x_{n},z)\le \frac{s^2d(x_0,x_1)}{1-\alpha}\cdot \alpha^n,\quad n\in\Nat\,,
\end{equation}
holds for   any $0<\alpha<1.$
\end{enumerate}
  \end{prop}\begin{proof}
   Although, as we have remarked,  the first statement of the proposition follows from Theorem  \ref{t1.Czerw}, we show a proof based on Theorem \ref{t.Stemp}.
     Our presentation  follows  \cite{an-dung15b}.

Suppose that $d$ satisfies the $s$-relaxed triangle inequality, for some $s\ge 1.$ If $0<p\le 1$ is given by the equation $(2s)^p=1$,
  then,  by Theorem \ref{t.Stemp}, the functional $\rho_p$ given by \eqref{def.p-metric} is a metric on $X$ satisfying the inequalities
\begin{equation}\label{ineq1.dung}
  \rho_p\le d^p\le 2\rho_p\,.\end{equation}

For  $x,y\in X$ let $x=x_0,x_1,\dots,x_n=y$ be an arbitrary chain in $X$ connecting $x$ and $y$. Then $y_i=f(x_i),\, i=0,1,\dots,n,$ is a chain in $X$ connecting
 $f(x)$ and $f(y)$. Consequently, by  \eqref{def.p-metric} and \eqref{eq1.Bahtin},
 \begin{equation}\label{ineq2.dung}
   \rho_p(f(x),f(y))\le \sum_{i=0}^{n-1}d(y_i,y_{i+1})^p\le \alpha^p \sum_{i=0}^{n-1}d(x_i,x_{i+1})^p.
 \end{equation}

 Since the inequality between the extreme terms in \eqref{ineq2.dung} holds for all chains $x=x_0,x_1,\dots,x_n=y,\, n\in\Nat,$   connecting $x$ and $y$,
 it follows
 $$
 \rho_p(f(x),f(y))\le\alpha^p\rho_p(x,y)\,,$$
 for all $x,y\in X$, where $0<\alpha^p<1.$ Consequently,  $f$ is a contraction with respect to  $\rho_p$. The inequalities \eqref{ineq1.dung} and the completeness of $(X,d)$ imply the completeness of $(X,\rho_p)$ and so, by Banach's contraction principle, $f$ has  a unique fixed point $z\in X$ and the sequence of iterates $(f^n(x))_{n\in\Nat}$ is $\rho_p$-convergent to $z$, for every $x\in X$. Appealing again to the inequalities \eqref{ineq1.dung}, it follows that $(f^n(x))_{n\in\Nat}$ is also $d$-convergent to $z$ for every $x\in X$.

 1.\;   The proof is similar to that of Banach's contraction principle in the metric case.  Observe first that, \eqref{eq1.Bahtin} implies
  \begin{equation}\label{eq2.Bahtin}
  d(f^n(x),f^n(y))\le\alpha^nd(x,y)\,,
  \end{equation}
  for all $n\in\mathbb{N}$ and $x,y\in X$.

  For  $x_0\in X$ consider the sequence of iterates
   $$x_n=f(x_{n-1})=f^n(x_0),\quad n\in\mathbb{N}\,.$$

   Let us prove that $(x_n)$ is a Cauchy sequence.

  By \eqref{s-relax-n} and  \eqref{eq2.Bahtin},
  \begin{equation}\label{ineq3.Bahtin}\begin{aligned}
  d(x_{n},x_{n+k+1})&\le sd(x_{n},x_{n+1}) + s^2d(x_{n+1},x_{n+2})+\dots\\&+s^{k}d(x_{n+k-1},x_{n+k})+s^{k}d(x_{n+k},x_{n+k+1})\\
  &\le \big(\alpha^{n}s+\alpha^{n+1}s^2+\dots+\alpha^{n+k-1}s^k\big)d(x_0,x_1)+\alpha^{n+k}s^{k}d(x_0,x_1)\\
  &=\alpha^{n}s\left(\frac{1-(\alpha s)^k}{1-\alpha s} +\alpha^ks^{k-1}\right)d(x_0,x_1)\\&=\alpha^{n}s
  \frac{1-(\alpha s)^{k-1}\alpha(s+1-\alpha s)}{1-\alpha s}d(x_0,x_1)\\&<\alpha^n\frac{s d(x_0,x_1)}{1-\alpha s}\,,
  \end{aligned}\end{equation}
  for all $n,k\in\mathbb{N}.$ Since $\lim_{n\to\infty}\alpha^{n}=0,$ this shows that $(x_n)$ is a Cauchy sequence. By the completeness of $(X,d)$ there exists $z\in X$ such that
  $\lim_{n\to\infty}d(x_n,z)=0$. We have
  \begin{align*}
    d(z,f(z))&\le sd(z,x_{n+1})+sd(x_{n+1},f(z))\\
    &\le sd(z,x_{n+1})+s\alpha d(x_{n},z)\lra 0\;\mbx{ as }\; n\to \infty\,.
  \end{align*}
  
  Hence $d(z,f(z))=0$ and so  $z=f(z).$
  
  Taking into account \eqref{ineq3.Bahtin},
    \begin{align*}
     d(x_{n},z)&\le s d(x_{n},x_{n+k+1})+s d(x_{n+k+1},z)\\
     &<\alpha^n\frac{s^2 d(x_0,x_1)}{1-\alpha s}+s d(x_{n+k+1},z)\,.
      \end{align*}
      
      Letting $k\to\infty$, one obtains \eqref{ineq2.Bahtin}.
    
   Suppose now that there exists two points $z,z'\in X$ such that $f(z)=z$ and $f(z')=z'$. Then the relations
  $$
  d(z,z')=d(f(z),f(z'))\le\alpha d(z,z')
  $$
  show that $d(z,z')=0$, i.e. $ z=z'.$

2.\; Let $x_0\in X$ and $x_n=f(x_{n-1}),\,n\in\mathbb{N}.$ Taking into account  the relaxed polygonal inequality and \eqref{eq2.Bahtin}, we obtain
\begin{align*}
  d(x_{n},x_{n+k})&\le s\sum_{i=0}^{k-1}d(x_{n+i},x_{n+i+1})
  \le   s(\alpha^n+\alpha^{n+1}+\dots+\alpha^{n+k})d(x_0,x_1)\\
 &=  s\alpha^n\,\frac{1-\alpha^{k+1}}{1-\alpha}\cdot d(x_0,x_1)< \frac{sd(x_0,x_1)}{1-\alpha}\cdot\alpha^n.
\end{align*}

Based on these relations the proof goes as in case 1.
\end{proof}

   \begin{remark}
     The proof given here to statement 2 from Proposition \ref{c1.Bahtin} is simpler than that of Theorem 12.4 in \cite{Kirk-Shah}.
   \end{remark}

\begin{remark} The proofs given in \cite{czerw93} and \cite{Kirk-Shah} to  Theorem \ref{t1.Czerw}   go in the following way.

Let $x$ be a fixed element of $ X$ and  $\epsic >0$. By \eqref{def.phi}.(b) there exists $m=m_\epsic\in\Nat$ such that
 \begin{equation}\label{eq1.err}
 \vphi^{m}(\epsic)<\frac\epsic{2s}\,.\end{equation}

 One considers the sequence $x_k=f^{km}(x),\ k\in \Nat,$ and one shows that there exists $k_0\in\Nat$ such that
 \begin{equation}\label{ineq1.Cauchy}
 d(x_k,x_{k'})<2s\epsic\,,
 \end{equation}
 for all $k,k'\ge k_0.$  One affirms that the inequality \eqref{ineq1.Cauchy} shows that $(x_k)$ is a Cauchy sequence, which is not surely true, because the inequality is true only for this specific $\epsic.$

 Taking another $\epsic,$ say $0<\epsic'<\epsic,$ we find another number $m'=m_{\epsic'}$ (possibly different from $m$),  such that

 \begin{equation}\label{eq2.err}
 \vphi^{m'}(\epsic')<\frac{\epsic'}{2s}\,.\end{equation}

 The above procedure yields a sequence $x'_k=f^{km'}(x),\, k\in \Nat,$ satisfying, for some   $k_1\in\Nat,$
 \begin{equation}\label{ineq2.Cauchy}
 d(x_k,x_{k'})<2s\epsic'\,,
 \end{equation}
 for all $k,k'\ge k_1.$

 But the sequences $(x_k)$ and $(x'_k)$ can be totally different, so we cannot  infer that the sequence $(x_k)$ is Cauchy.

As we have shown this flaw was fixed in the paper \cite{kaj-luk18}.\end{remark}
\begin{remark}
  Berinde \cite{beri93}  considers comparison functions satisfying  a  condition stronger than  $0<\alpha<1/s,$
  namely  $\sum_{k=1}^\infty \vphi^k(t)<\infty$, allowing estimations of the order of convergence similar to \eqref{ineq2.Bahtin}. He also shows that the sequence $x_n=f^n(x_0),\, n\in\Nat_0,$ is convergent to a fixed point of $f$ if and only if it is bounded.   For various kinds  of   comparison functions, the relations between them an applications to fixed points, see   \cite[\S3.0.3]{Rus-PP}.

\end{remark}

\subsection{Fixed points in generalized b-metric spaces}
 Theorem \ref{t1.Czerw} admits the following extension to    generalized b-metric spaces.

\begin{theo}\label{theorem-5}
Let $(X,d)$ be a complete generalized  b-metric space and suppose that the mapping $f: X \to X$ is such that

\begin{equation}\label{zad6}
d\left(f(x),f(y)\right) \le \varphi\left(d(x,y)\right)\,,
\end{equation}
for all $ x,y \in X$ with $d(x,y)<\infty,$ where the function $\varphi:\mathbb{R}_+\to \mathbb{R}_+$  satisfies the conditions (a),(b) from \eqref{def.phi}.

 Consider, for some $x \in X$, the sequence of successive approximations $\left(f^n (x) \right)_{n \in \mathbb{N}_0}$.
Then either

  {\rm (A) }\;\; $ d (f^k(x), f^{k+1}(x)) = + \infty $  for all $k \in \mathbb{N}_0$,\\
 or

    {\rm (B) }\;\; the sequence $\left(f^n (x) \right)_{n \in \mathbb{N}}$ is convergent to a fixed point of $f$.
\end{theo}

\begin{proof} Let $X=\bigcup_{i\in I}X_i$ be the canonical decomposition of $X$ corresponding to the equivalence relation \eqref{zad3}.
Assume that (A) does not hold. Then
$$ d(f^m(x), f^{m+1}(x)) < + \infty\,,$$
 for some $m \in \mathbb{N}_0\,.$
If $i\in I$ is such that    $ f^m(x), f^{m+1}(x) \in X_{i},$ then
$$
d\left( f^{m+1}(x), f^{m+2}(x)\right) \le \varphi\left( d\left( f^{m}(x), f^{m+1}(x)\right)\right)<\infty\,,$$
implies $f^{m+2}(x)\in X_i\,,$ and so, by mathematical induction,   $f^{m+k}(x)\in X_i$ for all  $k\in\mathbb N_0\,.$

Since
$$
  z\in X_i\iff d(z,f^m(x))<\infty\,,$$
the inequality
$$
d(f(z),f^{m+1}(x))\le \varphi(d(z,f^m(x))<\infty\,,$$
shows that the restriction $f_i=f|_{X_i}$ of $f$ to $X_i$ is a mapping from $X_i$ to $X_i$ satisfying
$$
d(f_i(y),f_i(z))\le\varphi(d(y,z))\,,$$
for all $y,z\in X_i.$

By Theorem \ref{t.2}, $X_i$  is  complete, so that,   by Theorem \ref{t1.Czerw},   $\left( f^{m+k}(x)\right)_{k\in\mathbb N_0}$ is convergent to a fixed point $z_i\in X_i$ of $f_{i}$, which  is a fixed point   for $f$.
\end{proof}

\begin{remark}
For $s = 1$  and $\varphi(t)=\alpha t,\, t\ge 0,$ where $0\le\alpha<1,$  we get the Diaz and Margolis \cite{diaz-margol68} fixed point theorem of the alternative. At the same  time this extends Theorem 2 from \cite{czerw-krol17} and  give  simpler proofs to  Theorems  2.1 and 3.1 from \cite{aydi-czerw18}.
\end{remark}

Proposition \ref{c1.Bahtin}  also admits   extensions to this setting as results of the alternative. We formulate only one of these results.

\begin{corol}\label{c2.Bahtin}  Let $(X,d)$ be a complete b-metric space, where $d$ satisfies the $s$-relaxed triangle inequality and let
$f:X\to X$ be a mapping satisfying, for some $ 0<\alpha<1,$ the inequality
\begin{equation*}
d(f(x),f(y))\le\alpha d(x,y)\,,
\end{equation*}
for all $x,y\in X$ with $d(x,y)<\infty.$ 

 Then, for every $x\in X,$ either

  {\rm (A$'$) }\;\; $ d (f^k(x), f^{k+1}(x)) = + \infty $  for all $k \in \mathbb{N}_0$,\\
 or

    {\rm (B$'$) }\;\; the sequence $\left(f^n (x) \right)_{n \in \mathbb{N}}$ is convergent to a fixed point of $f$.
  \end{corol}

  \section{Lipschitz functions}
  In this section we shall discuss the behavior  of Lipschitz functions defined on or taking values in  quasi-normed spaces and of Lipschitz functions on spaces of homogeneous type.

\subsection{Quasi-normed spaces}\label{Ss.quasi-normed space}
We start by a brief presentation of an important class of b-metric spaces -- quasi-normed spaces.
Good references  are \cite{kalt03}, \cite{Kalton-F-sp}, \cite[pp. 156-166]{Kothe}, \cite{Rolew-MLS}.

A \emph{quasi-norm} on a vector space $X$ (over $\Kapa$ equal to $\Real$ or $\Complex$) is a functional $\|\cdot\|:X\to\mathbb{R}_+$ for which  there exists a real number $k\ge 1$ so that
\begin{enumerate}
\item[(QN1) ] $\|x\|=0\iff x=0;$
\item[(QN2) ] $\|\alpha x\|=|\alpha|\,\| x\|;$
\item[(QN3) ] $\|x+y\|\le k(\| x\|+\|y\|),$
\end{enumerate}
for all $x,y\in X $ and $\alpha \in \mathbb{K}$. The pair $(X,\norm)$ is called a quasi-normed space. A complete quasi-normed space is called a quasi-Banach space.

If $k=1$, then $\|\cdot\|$ is a norm. The smallest constant $k$ for which the inequality (QN3) is satisfied for all $x,y\in X$ is called the \emph{modulus of concavity} of the quasi-normed space $X.$

For a  linear operator $T$ from a quasi-normed   space $(X,\|\cdot\|_X)$ to a normed space $(Y,\|\cdot\|_Y)$ put
$$
\|T\|=\sup\{\|Tx\|_Y : x\in X, \,\|x\|_X\le 1\}\,.$$

In particular,
\begin{equation}\label{norm-quasi-Banach}
\|x^*\|=\sup\{|x^*(x)| : x\in X,\|x\|_X\le 1\},\;\; x^*\in X^*,
\end{equation}
is a norm on the dual space $X^*=(X,\|\cdot\|_X)^*$.

It follows that $T$ is continuous if and only if $\|T\|<\infty$ and, in this case,
\begin{equation*}
\|Tx\|_Y\le\|T\|\|x\|_X\,,\;\; x\in X,
\end{equation*}
$\|T\|$ being the smallest number $L\ge 0$  for which the inequality $\|Tx\|_Y\le L\|x\|_X$ holds for all $x\in X$. If $Y$ is also a quasi-normed space, then  $\|\cdot\|$ is only a quasi-norm on the space $\mathcal L(X,Y)$ of all continuous linear operators from $X$ to $Y$.

 Two quasi-norms $\Vert \cdot\Vert_ 1,\Vert \cdot\Vert_ 2$ on a vector space $X$ are called \emph{equivalent} if they generate the same topology, or equivalently, if

$$\Vert x_n-x\Vert_ 1\to 0\iff  \Vert x_n-x\Vert_ 2\to 0\,,$$
for all sequences $(x_n)$ in $X$ and $x\in X$. As in the case of norms, the equivalence of two quasi-norms $\Vert \cdot\Vert_ 1,\Vert \cdot\Vert_ 2$ on a vector space $X$ is equivalent
to the existence of two numbers $\alpha,\beta>0$ such that
$$
\alpha\Vert x\Vert_ 1\le\Vert x\Vert_ 2\le\beta\Vert x\Vert_ 1\,,$$
for all $x\in X$.

A subset $A$ of a topological vector space (TVS) $(X,\tau)$ is called \emph{bounded} if it is absorbed by any 0-neighborhood, i.e for every $V\in\mathcal V_\tau(0)$ there exists $t>0$ such that $A\subseteq t V.$ A TVS is called \emph{locally bounded} if it has a bounded 0-neighborhood. A quasi-normed space $(X,\|\cdot\|)$ is locally bounded, as  the closed unit ball $B_X=\{x\in X : \|x\|\le 1\}$ is a bounded neighborhood of 0. One shows that, conversely, the topology of every locally  bounded TVS is generated by a quasi-norm.

A quasi-normed space $(X,\|\cdot\|)$ is normable (i.e. there exists a norm $\|\cdot\|_1$ on $X$ equivalent to the quasi-norm $\|\cdot\|$) if and only if 0 has a bounded convex neighborhood (implying that $X$ is locally convex).

\begin{defi}\label{def.F-norm}
 An $F$-\emph{norm} on a vector space $X$ is a mapping $\|\cdot\|:X\to\mathbb{R}_+$ satisfying the conditions  \index{$F$-norm}
\begin{enumerate}
\item[(F1) ] $\|x\|=0\iff x=0;$
\item[(F2) ] $\|\lambda x\|\le \|x\|$\;   for all \; $\lambda\in \mathbb{K}$ with $|\lambda|\le 1;$
\item[(F3) ] $\|x+y\| \le  \|x\| +\|y\|;$
\item[(F4) ] $\|x_n\|\to 0\;\Longrightarrow \|\lambda x_n\|\to 0;$     \index{F-norm}
\item[(F5) ] $\lambda_n\to 0\;\Longrightarrow \|\lambda_n x\|\to 0,$
\end{enumerate}
for all $ x, y, x_n \in X$ $\lambda,\lambda_n\in \mathbb{K}.$
An $F$-\emph{space} is a vector space equipped with a complete $F$-norm. \end{defi}\index{$F$-norm}

It follows that $d(x,y)=\|y-x\|,\, x,y\in X,$ is a translation-invariant metric on $X$ defining a vector topology. It is known that the metrizability of a TVS $(X,\tau)$ is equivalent to the existence of a countable basis of 0-neighborhoods, and in this case there exists a translation-invariant  metric $d$ on $X$ generating the topology $\tau$.

One shows, see \cite[p. 163]{Kothe},  that the topology of a metrizable TVS can be always given by an $F$-norm. If $(X,\tau)$ is a TVS, then the topology $\tau$ generates a uniformity $\mathcal W_\tau$ on $X $, a basis of it being given by the sets
$$
W_U=\{(x,y)\in X^2  : y-x\in U\}\,,$$
where $U$ runs over a 0-neighborhood basis in $X.$  Any translation-invariant metric generating the topology $\tau$  generates the same uniformity $\mathcal W_\tau$, so that
if $X$ is complete with respect to $\mathcal W_\tau$, then it is complete with respect to any translation-invariant metric generating the topology $\tau.$

\index{F-space}

Typical examples of quasi-normed  spaces are the spaces $L^p[0,1]$ and $\ell^p$ with $0<p<1$ equipped with the quasi-norms
\begin{equation}\label{def.Lp-qn}
\|f\|_p=\left(\int_0^1|f(t)|dt\right)^{1/p}\; \mbox{ and }\;   \|x\|_p=\left(\sum_{k=1}^\infty|x_k|^p\right)^{1/p}\,,
\end{equation}
for $f\in L^p[0,1]$ and $x=(x_k)_{k\in\mathbb{N}}\in\ell^p\,,$  respectively.

The quasi-norms $\|\cdot\|_p$ satisfy the inequalities
\begin{equation}\begin{aligned}\label{ineq1.Lp}
&\|f+g\|_p\le 2^{(1-p)/p}(\|f\|_p+\|g\|_p)\; \mbox{ and}\\
&\|x+y\|_p\le 2^{(1-p)/p}(\|x\|_p+\|y\|_p)\,,
\end{aligned}\end{equation}
for all  $f,g\in L^p[0,1]$ and $x,y\in\ell^p.$

The constant $2^{(1-p)/p}>1$ is sharp, i.e. the moduli of concavity of the spaces $L^p[0,1]$ and $\ell^p$ are both equal to $2^{(1-p)/p}$.

To show this, we  start with the elementary inequalities
\begin{equation}\label{ineq.Lp}
(a+b)^p\le a^p+b^p\le 2^{1-p}(a+b)^p\,,\end{equation}
valid for all $a,b>0$.

Let $f,g\in L^p[0,1].$ The first inequality from above implies
$$
|f(t)+g(t)|^p\le (|f(t)|+|g(t)|)^p\le  |f(t)|^p+|g(t)|^p\,,
$$
for almost all $t\in [0,1], $ so that
\begin{equation}\label{ineq3.Lp}\begin{aligned}
  \Vert f+g\Vert_ p^p&=\int_0^1|f(t)+g(t)|^pdt\le \int_0^1|f(t)|^pdt+\int_0^1|g(t)|^p dt \\&=\Vert f\Vert_ p^p+\Vert g\Vert_ p^p\,.
\end{aligned}\end{equation}

This inequality and the second inequality from \eqref{ineq.Lp} yield

\begin{align*}
  \Vert f+g\Vert_ p&=\left(\Vert f+g\Vert_ p^p\right)^{1/p}\le\left(\Vert f\Vert_ p^p+\Vert g\Vert_ p^p\right)^{1/p}\\
  &\le 2^{(1-p)/p}(\Vert f\Vert_ p+\Vert g\Vert_ p)\,.
\end{align*}

Similar calculations can be done to show that
$$
\Vert x+y\Vert_ p\le 2^{(1-p)/p}(\Vert x\Vert_ p+\Vert y\Vert_ p)\,,
$$
for all $x,y\in\ell^p.$

To show that the constant $2^{(1-p)/p}$ is sharp take $x=(1,0,0\dots)$ and $y=(0,1,0\dots)$ in the case of the space $\ell^p$. Then
$$
\Vert x+y\Vert_ p=2^{1/p}\; \mbox{ and }\; 2^{(1-p)/p}(\Vert x\Vert_ p+\Vert y\Vert_ p)=2^{(1-p)/p} \cdot 2=2^{1/p}\,,$$
that is, we have equality in the second inequality from \eqref{ineq1.Lp}. In the case of the space $L^p[0,1]$ take $f=\chi_{[0,\frac12)}$ and
$g=\chi_{[\frac12,1]}$ to obtain equality in the first inequality from \eqref{ineq1.Lp}.

\begin{remark}
  Apparently similar, the quasi-normed spaces $\ell^p$ and $L^p[0,1]$ drastically differ. For instance, the space $L^p[0,1]$ has trivial dual, $(L^p[0,1])^*=\{0\}$, while $(\ell^p)^*=\ell^\infty$, see \cite[pp. 156-158]{Kothe}.   D. Pallaschke \cite{pallas73} and Ph. Turpin \cite{turpin73}) have shown that every compact endomorphism of $L^p,\, 0<p<1,$ is null. N. Kalton and J. H. Shapiro \cite{kalt-shapir75} showed that there exists a quasi-Banach space with trivial dual admitting non-trivial compact endomorphisms. The example is a quotient space of $H^p,\, 0<p<1.$
Here,  $H^p,\, 0<p<1,$ denotes the classical Hardy quasi-Banach spaces of analytic functions in the unit disk of $\Complex$.
\end{remark}

  A $p$-\emph{norm}, where $0<p\le 1,$ is a mapping $\|\cdot\|:X\to\mathbb{R}_+$ satisfying (QN1) and (QN2) and
  \begin{enumerate}
  \item[(QN$3'$) ]
  ${\displaystyle \|x+y\|^p\le \|x\|^p+\|y\|^p\,,}$
  for all $x,y\in X$.\end{enumerate}
The quasi-norms of the spaces  $L^p[0,1]$ and  $\ell^p,\, 0<p<1,$ a $p$-norms, i.e.
\begin{align*}
  &\|f+g\|^p_p\le  \|f\|^p_p+\|g\|^p_p\;\mbox{ and}\\
  &\|x+y\|^p_p\le  \|x\|^p_p+\|y\|^p_p,
\end{align*}
for all  $f,g\in L^p[0,1]$ and  $x,y\in \ell^p.$

  A famous result of T. Aoki \cite{aoki42} and S. Rolewicz \cite{rolew57} says that on any quasi-normed space $(X,\|\cdot\|)$ there exists a $p$-norm equivalent to $\|\cdot\|$, where $p$ is determined from the  equality $2^{1/p}=k,$ $k$ being the constant from (QN3).

  Let $0<p\le 1.$ A subset $A$ of a vector space $X$ is called $p$-\emph{convex} if $\alpha x+\beta y\in A$ for all $x,y\in A$ and all $\alpha,\beta\ge 0$ with $\alpha^p+\beta^p=1, $ and $p$-\emph{absolutely convex} if $\alpha x+\beta y\in A$ for all $x,y\in A$ and all $\alpha,\beta\in\mathbb{K}$ with $|\alpha|^p+|\beta|^p\le 1. $
For $p=1$ one obtains the usual  convex and absolutely convex sets, respectively.

A TVS $X$ is $p$-normable if and only if it has a bounded $p$-convex 0-neighborhood, see \cite[p. 161]{Kothe}. One shows first that under this hypothesis there exists a bounded $p$-absolutely convex neighborhood $V$ of 0 and one defines the $p$-norm as the Minkowski functional corresponding to $V$, i.e. $\|x\|=\inf\{t : t>0,\, x\in tV\}.$\index{$p$-norm}\index{$p$-normed space}\index{set!$p$-convex}\index{set!$p$-absolutely convex}

\begin{remark}
    In K\"othe \cite{Kothe} by a $p$-norm on a vector space $X$ one understands a mapping $\Vert \cdot\Vert':X\to\mathbb{R}_+$ such that
    $$
 \Vert x\Vert '=0 \iff x=0, \quad \Vert \alpha x\Vert '=|\alpha|^p\Vert x\Vert '\;\;\mbox{ and }\;\;  \Vert x+y\Vert '\le \Vert x\Vert '+\Vert y\Vert '\,,$$
    for all $x,y\in X$ and $\alpha\in\mathbb{K}$.  In this case the ``$p$-norm" corresponding to a bounded  absolutely $p$-convex 0-neighborhood is given by
    $\Vert x\Vert '=\inf\{t^p : t>0,\, x\in tV\}.$

  It follows  that  $\Vert \cdot\Vert $ is a $p$-norm in the sense given here if and only if $\Vert \cdot\Vert ^p$ is a $p$-norm in the sense given in \cite{Kothe}.
  \end{remark}

\medskip

\textbf{The Banach envelope.}\index{quasi-Banach space!Banach envelope}

Let $(X,\|\cdot\|)$ be a quasi-Banach space and $B_X=\{x\in X : \|x\|\le 1\}$ its closed unit ball. Denote by $\|\cdot\|_C$ the Minkowski functional of the set $C=\conv(B_X)$.
It is obvious that $\|\cdot\|_{C}$ is a seminorm on $X$ and a norm on the quotient space $X/N$, where $N=\{x\in X :\|x\|_C=0\}. $  Since, for  $x\ne 0,\, x':=x/\|x\|\in B_X\subseteq C,$ it follows $\|x'\|_C\le 1,$ that is  $\|x\|_C\le \|x\|.$ Denote by $\widehat X$ the completion of $X/N$ with respect to the quotient-norm $\|\cdot\|_{\widehat X}$ corresponding to $\|\cdot\|_C$, whose (unique) extension to $\widehat X$ is denoted also by $\|\cdot\|_{\widehat X}$.
It follows $\|\widehat x\|_{\widehat X}\le \|x\|$ for all $x\in X$, hence the embedding $j:X\to\widehat X$ is continuous and one shows that $j(X)$ is dense in $\widehat X$.  The space $\widehat X$ is called the \emph{Banach envelope} of the quasi-Banach space $X$.

We distinguish two situations.

\smallskip

I.  $X$ \emph{has trivial dual}: $X^*=\{0\}$.

In this case $C=\conv(B_X)=X$ (see \cite[Proposition 2.1, p. 16]{Kalton-F-sp}) and so  $\|\cdot\|_C\equiv 0,\ N=X$  and $X/X=\{\widehat 0\}$.
It follows $\widehat X=\{0\}$ and $\widehat X^*=\{0\}=X^*.$ In particular $\widehat{L^p}=\{0\}$, where $L^p=L^p[0,1].$

\smallskip

II.  $X$ \emph{has a separating dual}.

This means that for every $x\ne 0$ there exists $x^*\in X^*$ with $x^*(x)\ne 0$  (e.g. $X=\ell^p$ with $0<p<1$).
In this case $\|\cdot\|_C$ is a norm on $X$ which can be calculated  by the formula
\begin{equation}\label{eq.quasi-Banach-envel}
\|x\|_C=\sup\{|x^*(x)| : x^*\in X^*,\, \|x^*\|\le 1\}\,.\end{equation}
where the norm of $x^*\in X^*$ is given by \eqref{norm-quasi-Banach}.

Consequently $N=\{0\}$, $X/N=X$ and we can consider $X$ as a dense subspace of $\widehat X$ (in fact, continuously and densely embedded in $\widehat X$).

It follows that:
\begin{enumerate}
\item[\rm (i) ] every continuous linear functional  on $(X,\|\cdot\|)$ has a unique norm preserving extension to $(\widehat X,\|\cdot\|_{\widehat X})$;
\item[\rm (ii) ] every continuous linear operator $T$ from $(X,\|\cdot\|)$ to a Banach space $Y$ has a unique norm preserving extension  $\widehat T:(\widehat X,\|\cdot\|_{\widehat X})\to Y$. \end{enumerate}

Consequently $(X,\|\cdot\|)^*$  can be identified with $(\widehat X,\|\cdot\|_{\widehat X})^*$ and   the norm $\|\cdot\|_{\widehat X}$ can also be calculated by the formula \eqref{eq.quasi-Banach-envel} for all $x\in \widehat X$.

One shows that the Banach envelope of $\ell^p$ is $\ell^1$,  for every $0<p<1$.

Another way to define the Banach envelope in the case of a quasi-Banach space with separating dual is via the embedding $j_X$ of $X$ into its bidual $X^{**}$ (see \cite{Kalton-F-sp}). Since $X^*$ separates the points of $X$, it follows that $j_X$ is injective of norm  $\|j_X\|\le 1$ (in this case one can not prove that $\|j_X\|= 1$ because the Hahn-Banach extension theorem may fail for non-locally convex spaces).

By \eqref{eq.quasi-Banach-envel}
$$
\|x\|_C=\sup\{|x^*(x)| : \|x^*\|\le 1\}=\|j_X(x)\|_{X^{**}}\,,$$
so we can identify $\widehat X$ with the closure of $j_X(X)$ in $(X^{**},\|\cdot\|_{X^{**}})$

\subsection{Lipschitz functions and quasi-normed spaces}

It turns out that some results concerning Banach space-valued Lipschitz functions fail in  the quasi-Banach case and, in some cases,  the validity of some of them forces the quasi-Banach space  to be locally convex, i.e.  a Banach space.

In this subsection we consider only spaces over $\Real.$

Let  $(Z,d)$ be a b-metric space and $(Y,\|\cdot\|)$  a quasi-normed space.  A function $f:Z\to  Y$ is called Lipschitz if there exists $L\ge 0$ (called a Lipschitz constant for $f$) such that
 \begin{equation*}
 \|f(z)-f(z')\|\le L d(z,z')\,,
 \end{equation*}
 for all $z,z'\in Z$. One denotes by $\Lip(Z,Y)$ the   space of all   Lipschitz functions from $Z$ to $Y$.

The \emph{Lipschitz norm}  $\|f\|_L$ of  $f $ is  defined by
$$
\|f\|_L=\sup\left\{\frac{\|f(z)-f(z')\|}{d(z,z')} : z,z'\in Z,\, z\ne z'\right\}\,.$$

It follows that $\|f\|_L$ is the smallest Lipschitz constant for $f$.

Since $\|f\|=0$ if and only if $f=$ const, $\norm$ is actually only a seminorm on $\Lip(Z,Y)$.  To obtain a norm, one considers  a fixed element $z_0\in Z$ and the space
$$\Lip_0(Z,Y)=\{f\in\Lip(Z,Y) : f(z_0)=0\}\,.$$

If $Z=X$, where $X$ is a quasi-normed space, then one take $0$ for the fixed point $z_0$, and, in this case,  $\Lip_0(X,Y)$  is a quasi-normed space, that is,
$$
\|f+g\|_L\le k(\|f\|_L+\|g\|_L)\,,$$
for $f,g\in \Lip_0(X,Y)$, where $k\ge 1$ is the constant from (QN3). It is complete, provided $Y$ is a quasi-Banach space.

If $Y=\Real$, then one uses the notation $\Lip(X),\, \Lip_0(X)$ and $\Lip_0(X)$ is called the Lipschitz dual of the quasi-normed space $X$.

We noted  that the space $L^p=L^p[0,1]$ has trivial dual. F. Albiac \cite{albiac08} proved that it has also a trivial Lipschitz dual, i.e. $\Lip_0(L^p)=\{0\}.$ Later he showed that this is a more general phenomenon.

\begin{prop}[F. Albiac \cite{albiac11}]Let $(X,\|\cdot\|)$ be  a quasi-Banach space and
$$
|||x|||:=\sup\{f(x) : f\in\Lip_0(X),\, \|f\|_L\le 1\}\,, x\in X\,.$$

Then
\begin{enumerate}
\item[\rm (i)]\;\; $|||\cdot|||$ is a seminorm on $X$;
\item[\rm (ii)]\;\; if $\Lip_0(X)$ is nontrivial, then $X$ has a nontrivial dual, i.e. $X^*\ne\{0\}$;
\item[\rm (iii)]\;\; if $X$ has a separating Lipschitz dual, then $X$ has a separating (linear) dual and $|||\cdot|||$ is a norm on $X$.
\end{enumerate}\end{prop}

One says that a family $\rd F$ of real valued functions on a quasi-normed  space $X$ is separating if for every $x\ne y$ in $X$ there exists $f\in\rd F$ with $f(x)\ne f(y).$

It is known that every Lipschitz function  $f$ from a subset   of a metric space $(X,d)$ to $\Real$ admits an extension to $X$ which is Lipschitz with the same Lipschitz constant  (McShane's extension theorem).  The following result shows that the validity of this result for every subset of a quasi-Banach space $X$ forces this space to be Banach.

\begin{prop}[\cite{albiac11}] Let $(X,\|\cdot\|)$ be  a quasi-Banach space. If for every subset $Z$ of $X$, every $L$-Lipschitz function $f:Z\to\mathbb{R}$ admits an $L'$-Lipschitz extension, for some $L'\ge L$, then the space $X$ is locally convex, i.e. it is a Banach space.
 \end{prop}

 It is known that every continuous linear operator from a quasi-Banach space $X$ to a Banach space $Y$ admits a norm preserving linear extension to the Banach envelope $\widehat X$ of $X$ to $Y$.  F. Albiac \cite{albiac11} has shown that this is true for Lipschitz mappings too: every Lipschitz mapping $f:X\to Y$ admits a unique Lipschitz  extension with the same Lipschitz constant  $\widehat f:\widehat X\to Y$.

Moreover, if $X,Y$ are normed spaces and $f:X\to Y$ is G\^ateaux differentiable on the interval $[x,y]:=\{x+t(y-x) : t\in[0,1]\},$ then
\begin{equation}\label{eq1.G-MVT}
\|f(x)-f(y)\|\le \|x-y\|\sup\{\|f'(\xi)\| :\xi\in[x,y]\}\,.\end{equation}
\begin{prop}[\cite{albiac11}] Let $(X,\|\cdot\|)$ be  a quasi-Banach space. If every nonconstant G\^ateaux differentiable Lipschitz function $f:[0,1]\to X$ satisfies the mean value inequality \eqref{eq1.G-MVT} for all $x,y\in [0,1]$, then the space $X$ is locally convex, i.e. it is a Banach space.
\end{prop}

Let  $\alpha>0.$ A function $f:(X_1,d_1)\to(X_2,d_2)$ between two b-metric spaces $X_1,X_2$ is called \emph{H\"older of order} $\alpha$ if there exists $L\ge 0$ such that
\begin{equation}\label{def.Hold-fcs}
d_2(f(x),f(y))\le Ld_1(x,y)^\alpha\,,
\end{equation}
for all $x,y\in X_1.$  As a consequence of the mean value theorem, every function $f$ from $[0,1]$ to a Banach space $X$ which is H\"older of order $\alpha>1$
 is constant, a fact that is no longer true if $X$ is a quasi-Banach space.
\begin{example} Let $L^p=L^p[0,1]$ for $0<p<1.$  The function $f:[0,1]\to L^p$ given by $f(t)=\chi_{[0,t]}$ satisfies the equality
$$
\|f(s)-f(t)\|_p =|s-t|^{1/p}\,,$$
for all $s,t\in [0,1]$, where $\|\cdot\|_p$ is the $L^p$-norm (see \eqref{def.Lp-qn}).
\end{example}
Indeed, for $0\le t<s\le 1,$
$$
\|f(s)-f(t)\|_p=\left(\int_t^s\chi^p_{(t,s]}(u) du\right)^{1/p}=|s-t|^{1/p}\,.$$

The Riemann integral of  a function $f:[a,b]\to X$, where $[a,b]$ is an interval in $\mathbb{R}$ and  $X$ is a Banach space, can be defined as in the real case, by simply replacing the absolute value $|\cdot|$ with  the norm sign $\|\cdot\|$, and has properties similar to those from the real case. For instance, the following result is true.

\begin{prop}[\cite{albiac11}] Let $X$ be a Banach space.
 If $f:[a,b]\to X$ is continuous, then
  \begin{enumerate}
  \item[\rm(i) ] $f$ is Riemann integrable, and
  \item[\rm(ii) ] the function \begin{equation}\label{eq1.R-int}
  F(t)=\int_a^tf(s)ds,\, t\in[a,b]\,,\end{equation}
   is differentiable with $F'(t)=f(t)$ for all $t\in[a,b]$.
  \end{enumerate}\end{prop}

\begin{remark} However, there is a point  where this analogy is broken, namely the Lebesgue criterion of Riemann integrability: a function $f:[a,b]\to\mathbb{R}$ is Riemann integrable if and only if it is continuous almost everywhere on $[a,b]$ (i.e. excepting a set of Lebesgue measure zero). In the infinite dimensional  case this criterion does not hold in general, leading to the study  of those Banach spaces for which it, or some weaker forms, are true , see, for instance, \cite{gordon91}, \cite{sofi12}, \cite{sofi16b} and the references quoted therein.
\end{remark}

In the case of quasi-Banach spaces the situation is different. By a result attributed to S. Mazur and W. Orlicz \cite{maz-orlicz48} (see also \cite[p. 122]{Rolew-MLS}) an $F$-space $X$ is locally convex if and only if every continuous function $f:[0,1]\to X$ is Riemann integrable.

M. M. Popov \cite{popov94} investigated the Riemann integrability of functions defined on intervals in $\mathbb{R}$ with values in an $F$-space. Among other results, he proved that a Riemann  integrable function $f:[a,b]\to X$  is bounded and that the function $F$ defined by \eqref{eq1.R-int} is uniformly continuous, but there exists a continuous function $f:[0,1]\to\ell^p$, where $0<1<p$, such that the function $F$ does not have a right derivative at $t=0.$ He asked whether any continuous function $f$ from $[0,1]$ to $L^p[0,1], \,0<p<1,$ (or more general, to a quasi-Banach space $X$ with $X^*=\{0\}$) admits a primitive. This problem was solved by N. Kalton \cite{kalton96} who proved that if $X$ is a quasi-Banach space with $X^*=\{0\}$, then every continuous function $f:[0,1]\to X$ has a primitive. Kalton considered the space $C^1_{Kal}(I,X),\,$ where $ I=[0,1]$ and $X$ is a quasi-Banach space, of all continuously differentiable functions $f:I\to X$ such that the function $\tilde f:I^2\to X$ given for $s,t\in I$ by
$\tilde f(t,t)=f'(t)$ and $\tilde f(s,t)=(f(s)-f(t))/(s-t)$ if $s\ne t$, is continuous. It follows that  $C^1_{Kal}(I,X)$  is a  quasi-Banach space with respect to the quasi-norm
$$
\|f\|=\|f(0)\|+\|f\|_L\,.$$

The notation $C^1_{Kal}(I,X)$ was introduced in \cite{albiac-ans12a}; Kalton used the notation $C^1(I;X)$.

Denote by $C(I,X)$ the Banach space (with respect to the sup-norm) of all continuous functions from $I$ to $X$. The \emph{core} of a quasi-Banach space $X$ is the  maximal subspace $Z$ of $X$ (denoted by core$(X)$) with $Z^*=\{0\}$. One shows that such a subspace always exists, is unique and closed. Notice that core$(X)=\{0\}$  implies only that $X$ has a nontrivial dual, but not necessarily a separating one.

In \cite{albiac-ans12b} it is shown that if $X$ is a quasi-Banach space with core$(X)=\{0\}$, then there exists  a continuous function $f:[0,1]\to X$ failing to have a primitive.

Kalton, \emph{op. cit}., called a quasi-Banach $X$ a $D$-\emph{space} if the mapping  $$D:C^1_{Kal}(I,X)\to C(I,X)\,,$$ given by $Df=f',$    is surjective and proved the following result.
\begin{theo}[\cite{kalton96}] Let $X$ be a quasi-Banach with core$(X)=\{0\}$. Then $X$ is a $D$-space if and only if $X$ is locally convex (or, equivalently, a Banach space).
\end{theo}

It is known that every continuously differentiable function from an interval $[a,b]\subseteq\mathbb{R}$ to a Banach space $X$ is Lipschitz with $\|f\|_L=\sup\{\|f'(t)\|:t\in[a,b]\}$ (a consequence of the Mean Value Theorem, see \eqref{eq1.G-MVT}). As it was shown in \cite{albiac-ans12a} this in no longer true in quasi-Banach spaces.
\begin{theo} Let $X$ be a non-locally convex quasi-Banach space $X$. Then there exists a function $F:I\to X$ such that:
\begin{enumerate}
\item[\rm (i) ] $F$ is continuously differentiable on $I$;
\item[\rm (ii) ] $F'$ is Riemann integrable on $I$ and $F(t)=\int_0^tF'(s)ds,\, t\in I$;
\item[\rm (iii) ] $F$ is not Lipschitz  on $I$.\end{enumerate}
\end{theo}

In \cite{albiac-ans13} it is proved that the usual rule of the calculation of the integral (called Barrow's rule by the authors, known also as Leibniz rule) holds in the quasi-Banach case in the following form.
\begin{prop} Let $X$ be a quasi-Banach with separating dual. If $F:[a,b]\to X$ is differentiable with Riemann integrable derivative, then
$$
\int_a^bF'(t) dt=F(b)-F(a)\,.$$
\end{prop}

Another pathological result concerning differentiability of quasi-Banach valued Lipschitz functions was obtained by N. Kalton \cite[Theorem 3.3]{kalton81}.
\begin{theo}Let $X$ be an $F$-space with trivial dual. Then for every pair of distinct points $x_0,x_1\in X$ there exists a function $f:[0,1]\to X$ such that
$f(0)=x_0,\, f(1)=x_1$ and
$$
\lim_{|s-t|\to0}\frac{f(s)-f(t)}{s-t}=0\;\mbox{ uniformly for }\; s,t\in[0,1]\,.$$

In particular $f'(t)=0$ for all $t\in[0,1].$
\end{theo}

\begin{remark} N. Kalton \cite[Corollary 3.4]{kalton81} also remarked  that if $X$ is an $F$-space and $x\in X\setminus\{0\}$, then in order to exist  a function $f:[0,1]\to X$ such that  $f(0)=0$, $f(1)=x$  and $f'(t)=0$ for all $ t\in [0,1],$ it is necessary and sufficient that $x\in$ core$(X).$
\end{remark}

If $X$ is a Banach space and  $f:[0,1]\to X$ is continuous then it is Riemann integrable and the average function ${\rm Ave}[f]:[a,b]\times[a,b]\to X$, given by
$$
{\rm Ave}[f](s,t)=\begin{cases}
  \frac1{t-s}\int_s^tf(u)du \quad &\mbox{if }\; a\le s<t\le b,\\
  f(c)\quad & \mbox{if } \; s=t=c\in[a,b],\\
  \frac1{s-t}\int_t^sf(u)du \quad &\mbox{if }\;a\le t<s\le b,
\end{cases}$$
 is jointly continuous on $[a,b]\times[a,b]$, and so, separately continuous and bounded. Some pathological properties of the average function in the quasi-Banach case are examined in  \cite{albiac-ans12a}, \cite{albiac-ans14} and \cite{popov94}.

 The analog of the Radon-Nikod\'ym Property  for quasi-Banach spaces and its connections with the differentiability of Lipschitz mappings and martingales are discussed in \cite{albiac-ans16}.
\subsection{Lipschitz functions on spaces of homogeneous type}
Let $(X,d,\mu)$ be a space of homogeneous type (see Subsection \ref{Ss.homog-sp}).  By $B$ we  shall denote balls of the form $B(x,r).$ If $\vphi$ is a function integrable on bounded sets, then the \emph{mean value} of $\vphi$ on the ball $B$ is defined by
\begin{equation}\label{def.mv}
m_B(\vphi)=\mu(B)^{-1}\int_B\vphi(x)d\mu(x)\,.
\end{equation}

For $1\le q<\infty$ and $0<\beta<\infty$ one denotes by $\Lip(q,\beta)$ the set of all functions $\vphi$, integrable on bounded sets, for which there exists a constant $C\ge 0$ such that
\begin{equation}\label{eq1.maci-Lip}
\left(\frac{1}{\mu(B)}\int_B|\vphi(x)-m_B(\vphi)|^qd\mu(x)\right)^{1/q}\le C\mu(B)^\beta\,,
\end{equation}
for all balls $B$. The least constant $C$ for which \eqref{eq1.maci-Lip} holds will be denoted by $\|\vphi\|_{\beta,q}\,.$

We shall denote by $\Lip(\beta)$  the set of all functions $\vphi$ on $X$ such that there exists  a constant $C\ge 0$ satisfying
\begin{equation}\label{eq2.maci-Lip}
|\vphi(x)-\vphi(y)|\le Cd(x,y)^\beta\,,
\end{equation}
for all $x,y\in X$, i.e. it is H\"older of order $\beta$ (see \eqref{def.Hold-fcs}). The least $C\ge 0$ for which \eqref{eq2.maci-Lip} holds   is denoted by $\|\vphi\|_{\beta}\,.$

The following results concerning these classes of Lipschitz functions were proved in \cite{maci-sego79a}.
\begin{theo}
  Let $(X,d,\mu)$ be a space of homogeneous type. Then there exists a constant $C\ge 0$ (depending on $\beta$ and $q$ only) such that for every $\vphi\in\Lip(\beta,q)$ there exists a function $\psi$ satisfying
\begin{align*}
  &{\rm (i)}\;\;\quad \vphi(x)=\psi(x)\;  \mbox{ a.e. on }\; X,\;
\mbox{and} \\
&
{\rm (ii)}\;\, \quad |\psi(x)-\psi(y)|\le C\|\vphi\|_{\beta,q}\mu(B)^\beta\,,
\end{align*}
for any ball $B$ containing the $x,y$.
\end{theo}
\begin{theo}
  Let $(X,d,\mu)$ be a space of homogeneous type. Then, given $0<\beta<\infty$, there exists a b-metric $\delta$ on $X$ such that $(X,\delta,\mu)$ is a normal space of homogeneous type  and for every $1\le q<\infty$ we have
  \begin{equation*}
  \vphi\in\Lip(\beta,q)\;\mbx{ of }\; (X,d,\mu)\iff\exists\psi \in\Lip(\beta)\;\mbx{ of }\; (X,\delta,\mu)\;\mbx{with}\; \vphi\overset{{\rm a.e.}}{=}\psi.
  \end{equation*}

  Moreover, the norms $\|\vphi\|_{\beta,q}$ and $\|\psi\|_\beta$ are equivalent.
\end{theo}

\providecommand{\bysame}{\leavevmode\hbox to3em{\hrulefill}\thinspace}
\providecommand{\MR}{\relax\ifhmode\unskip\space\fi MR }
\providecommand{\MRhref}[2]{%
  \href{http://www.ams.org/mathscinet-getitem?mr=#1}{#2}
}
\providecommand{\href}[2]{#2}

\end{document}